\documentclass{amsart}
\usepackage{amssymb,amsmath,mathrsfs,amsfonts}
\usepackage{amscd, amssymb, amsmath, amsthm, graphics}
\usepackage{latexsym, graphics, graphicx, psfrag}
\usepackage{amsmath,amsfonts,amsthm,amssymb}

\usepackage[small,nohug,heads=vee]{diagrams}
\usepackage{fancyhdr}
\usepackage[all]{xy}
\usepackage[colorlinks=true]{hyperref}

\newtheorem{theorem}{Theorem}[section]
\newtheorem{lemma}[theorem]{Lemma}
\newtheorem{proposition}[theorem]{Proposition}
\newtheorem{corollary}[theorem]{Corollary}

\theoremstyle{definition}
\newtheorem{definition}[theorem]{Definition}

\newtheorem{question}[theorem]{Question}
\newtheorem{remark}[theorem]{Remark}

\begin{document}
\sloppy

\title[NonLERFness of mixed $3$-manifold groups]{NonLERFness of arithmetic hyperbolic manifold groups and mixed $3$-manifold groups}

\author{Hongbin Sun}
\address{Department of Mathematics, UC Berkeley, CA 94720, USA}
\curraddr{Department of Mathematics, Rutgers University - New Brunswick, Piscataway, NJ 08854, USA}
\email{hongbin.sun@rutgers.edu}

\thanks{}


\subjclass[2010]{57M05, 57M50, 20E26, 22E40}
\thanks{The author is partially supported by NSF Grant No. DMS-1510383.}
\keywords{locally extended residually finite, graph of groups, hyperbolic $3$-manifolds, arithmetic hyperbolic manifolds}

\date{\today}
\begin{abstract}
We will show that, for any noncompact arithmetic hyperbolic $m$-manifold with $m> 3$, and any compact arithmetic hyperbolic $m$-manifold with $m> 4$ that is not a $7$-dimensional arithmetic hyperbolic manifold defined by octonions, its fundamental group is not LERF. The main ingredient in the proof is a study on abelian amalgamations of hyperbolic $3$-manifold groups. We will also show that a compact orientable irreducible $3$-manifold with empty or tori boundary supports a geometric structure if and only if its fundamental group is LERF.
\end{abstract}

\maketitle
\vspace{-.5cm}
\section{Introduction}

For a group $G$ and a subgroup $H < G$, we say that $H$ is {\it separable} in $G$ if for any $g\in G\setminus H$, there exists a finite index subgroup $G'<G$ such that $H<G'$ and $g\notin G'$. $G$ is called {\it LERF} (locally extended residually finite) or {\it subgroup separable} if all finitely generated subgroups of $G$ are separable.

The LERFness of a group is a property closely related with low dimensional topology, especially the virtual Haken conjecture (settled in \cite{Ag3}). In this paper, we are mostly interested in fundamental groups of some nice manifolds, and graph of groups constructed from these groups.

Among fundamental groups of low dimensional manifolds, the following groups were known to be LERF: free groups (\cite{Ha}), surface groups (\cite{Sc}), Seifert manifold groups (\cite{Sc}), hyperbolic $3$-manifolds groups (\cite{Ag3} and \cite{Wi}); while the following groups are known to be nonLERF: the groups of nontrivial graph manifolds (\cite{NW2}), the groups of fibered $3$-manifolds whose monodromy is reducible and satisfies some further condition (\cite{Li1}).

In this paper, we give a few more examples of nonLERF groups arised from topology. These results imply that $3$-manifolds with LERF fundamental groups support geometric structures, and hyperbolic manifolds with LERF fundamental groups seem to have dimension at most $3$.

One main result of this paper is about high dimensional arithmetic hyperbolic manifolds (with dimension $\geq 4$). Comparing to the $3$-dimensional case, there are much fewer examples of hyperbolic manifolds with dimension at least $4$. Most examples of high dimensional hyperbolic manifolds are constructed by arithmetic methods, and some other examples are constructed by doing cut-and-paste surgery on these arithmetic examples. So the following results suggest that having nonLERF fundamental group is a general phenomenon in high dimensional hyperbolic world.

\begin{theorem}\label{arithmetic}
Let $M^m$ be an arithmetic hyperbolic manifold with $m\geq 5$ which is not a $7$-dimensional arithmetic hyperbolic manifold defined by octonions, then its fundamental group is not LERF.

Moreover, if $M$ is closed, there exists a nonseparable subgroup isomorphic to a free product of closed surface groups and free groups. If $M$ is not closed, there exists a nonseparable subgroup that is isomorphic to either a free subgroup, or a free product of surface groups and free groups.
\end{theorem}

Comparing with Theorem \ref{arithmetic}, it is shown in \cite{BHW} that all geometrically finite subgroups of standard arithmetic hyperbolic manifold groups are separable. It will be easy to see that nonseparable subgroups constructed in the proof of Theorem \ref{arithmetic} are not geometrically finite (Remark \ref{not_geom_finite}).

Theorem \ref{arithmetic} does not cover the case of arithmetic hyperbolic $4$-manifolds. By using a slightly different method, we show that noncompact arithmetic hyperbolic manifolds with dimension at least $4$ have nonLERF fundamental groups. Of course, the only case in Theorem \ref{noncompact} that is not covered by Theorem \ref{arithmetic} is the $4$-dimensional case. Note that in the more recent work \cite{Sun}, it is proved that all closed arithmetic hyperbolic $4$-manifolds also have nonLERF fundamental groups. So it is known that, with possible exceptions in $7$-dimensional arithmetic hyperbolic manifold defined by octonions, all arithmetic hyperbolic manifolds with dimension at least $4$ have nonLERF fundamental groups.

\begin{theorem}\label{noncompact}
Let $M^m$ be a noncompact arithmetic hyperbolic $m$-manifold with $m\geq 4$, then $\pi_1(M)$ is not LERF.

Moreover, there exist a nonseparable subgroup isomorphic to a free group and another nonseparable subgroup isomorphic to a surface group.
\end{theorem}

Some examples of high dimensional nonarithmetic hyperbolic manifolds are constructed in \cite{GPS}, \cite{Ag2} and \cite{BT}. These examples are constructed by cutting arithmetic hyperbolic manifolds along codimension-$1$ totally geodesic submanifolds, then pasting along isometric boundary components. Since all these nonarithmetic hyperbolic manifolds contain codimension-$1$ arithmetic hyperbolic submanifolds, Theorem \ref{arithmetic} implies Theorem \ref{cutpaste}, which claims that all nonarithmetic examples in \cite{GPS} and \cite{BT} (\cite{Ag2} only constructed $4$-dimensional examples) with dimension $\geq 6$ have nonLERF fundamental groups.

In Theorem \ref{reflection}, we also show that compact reflection hyperbolic manifolds with dimension $\geq 5$ and noncompact reflection hyperbolic manifolds with dimension $\geq 4$ have nonLERF fundamental groups.

\vspace{3mm}

Another main result in this paper concerns compact orientable irreducible $3$-manifolds with empty or tori boundary. Thurston's Geometrization Conjecture (confirmed by Perelman) implies that any compact orientable irreducible $3$-manifold $M$ with empty or tori boundary has a minimal collection of incompressible tori, such that each component of its complement supports one of eight Thurston's geometries. If this set of incompressible tori is empty, we say that $M$ is a geometric $3$-manifold.

The following theorem implies that a compact orientable irreducible $3$-manifolds with empty or tori boundary is geometric if and only if its fundamental group is LERF. The author thinks this result is very interesting, since it gives a surprising relation between geometric structures on $3$-manifolds and LERFness of $3$-manifold groups, and these two topics in $3$-manifold topology are very popular in the past twenty years. This result also confirms Conjecture 1.5 in \cite{Li1}.

\begin{theorem}\label{dim3}
For an compact orientable irreducible $3$-manifold $M$ with empty or tori boundary, $M$ supports one of eight Thurston's geometries if and only if $\pi_1(M)$ is LERF.

When $\pi_1(M)$ is not LERF, there exists a nonseparable subgroup isomorphic to a free group. If $M$ is a closed mixed $3$-manifold, there also exists a nonseparable subgroup isomorphic to a surface group.
\end{theorem}

The proof of Theorem \ref{dim3} is enlightened by the construction in Section 8 of \cite{Li1}. To prove this theorem, the main case we need to deal with is that $M$ is a union of two geometric $3$-manifolds along one torus, with one of them being hyperbolic.

From group theory point of view, the above group is a $\mathbb{Z}^2$-amalgamation of two LERF groups. An even simpler case is: a $\mathbb{Z}$-amalgamation of two hyperbolic $3$-manifold groups, i.e. the fundamental group of a union of two hyperbolic $3$-manifolds along one essential circle.

There have been a lot of works that study LERFness of $\mathbb{Z}$-amalgamated groups $A*_\mathbb{Z} B$, with both $A$ and $B$ being LERF. For example, the first nonLERF example of $A*_{\mathbb{Z}}B$ was constructed in \cite{Ri}. It has been shown that if both $A$ and $B$ are free groups (\cite{BBS}), or if $A$ is free, $B$ is LERF and $\mathbb{Z}<A$ is a maximal cyclic subgroup (\cite{Gi}), of if both $A$ and $B$ are surface groups (\cite{Ni1}), then $A*_\mathbb{Z}B$ is LERF.

Here we give a family of nonLERF $\mathbb{Z}$-amalgamations of $3$-manifold groups.

\begin{theorem}\label{singular}
Let $M_1, M_2$ be two finite volume hyperbolic $3$-manifolds, and $i_k:S^1\to M_k,\ k=1,2$ be two $\pi_1$-injective embedded circles, then the fundamental group of $$X=M_1\cup_{S^1}M_2$$ is not LERF.

Moreover, if both $M_1$ and $M_2$ have cusps, there exists a nonseparable subgroup isomorphic to a free group. If at least one of $M_k$ is closed, there exists a nonseparable subgroup isomorphic to a free product of surface groups and free groups.
\end{theorem}

Theorem \ref{singular} is the main ingredient to prove Theorem \ref{arithmetic}. We will use the fact that arithmetic hyperbolic manifolds have a lot of totally geodesic submanifolds of smaller dimension. If an arithmetic hyperbolic manifold has dimension at least $5$, there are two totally geodesic $3$-dimensional submanifolds intersecting along a closed geodesic, which gives a picture addressed in Theorem \ref{singular}.

In dimension $4$, such a picture does not show up because of dimension reason, so Theorem \ref{singular} does not help here. However, Theorem \ref{dim3} implies that the double of any cusped hyperbolic $3$-manifold has nonLERF fundamental group, and groups of all noncompact arithmetic hyperbolic manifolds with dimension $\geq 4$ contain such doubled manifold groups (by \cite{LR}). So Theorem \ref{noncompact} is a consequence of Theorem \ref{dim3}.

The organization of this paper is as the following. In Section \ref{preliminary}, we review some background on group theory, $3$-manifold topology and arithmetic hyperbolic manifolds. In Section \ref{sectiondim3}, we prove Theorem \ref{dim3}, which is enlightened by the construction in \cite{Li1}. In Section \ref{sectionsingular}, we prove Theorem \ref{singular}, whose proof is similar to the proof of Theorem \ref{dim3}, with some modifications. In Section \ref{sectionarithmetic}, we deduce Theorem \ref{arithmetic} and Theorem \ref{noncompact} from Theorem \ref{singular} and Theorem \ref{dim3} respectively. In Section \ref{further}, we ask some questions related to the results in this paper.

\subsection*{Acknowledgement}
The author is grateful to Ian Agol for many valuable conversations, and these conversations are very helpful on various aspects of the development of this paper. The author thanks Yi Liu for communication and explanation on his results in \cite{Li1}, and thanks Alan Reid for his help on arithmetic hyperbolic manifolds. Part of the work in this paper was done during the author's visiting at the Institute for Advanced Study, and the author thanks for the hospitality of IAS. The author also thanks the anonymous referee for many very helpful comments.

\section{Preliminaries}\label{preliminary}

In this section, we review some basic concepts in group theory, $3$-manifold topology and arithmetic hyperbolic manifolds.

\subsection{Locally extended residually finite}

In this subsection, we review basic concepts and properties on locally extended residually finite groups.

\begin{definition}
Let $G$ be a group, and $H<G$ be a subgroup, we say that $H$ is {\it separable} in $G$ if for any $g\in G\setminus H$, there exists a finite index subgroup $G'<G$ such that $H<G'$ and $g\notin G'$.
\end{definition}

An equivalent formulation is that $H$ is separable in $G$ if and only if $H$ is a closed subset under the profinite topology of $G$.

\begin{definition}
A group $G$ is {\it LERF} (locally extended residually finite) or {\it subgroup separable} if all finitely generated subgroups of $G$ are separable in $G$.
\end{definition}

A basic property on LERFness is that any subgroup of a LERF group is still LERF. This property is basic and well-known, and the proof is very simple. However, since we will use this property for many times in this paper, we give a proof here.

\begin{lemma}\label{subgroup}
Let $G$ be a group and $\Gamma<G$ be a subgroup. For a further subgroup $H<\Gamma$, if $H$ is separable in $G$, then $H$ is separable in $\Gamma$.

In particular, if $\Gamma$ is not LERF, then $G$ is not LERF.
\end{lemma}

\begin{proof}
We take an arbitrary element $\gamma\in \Gamma\setminus H$, then $\gamma\in G\setminus H$ holds. Since $H$ is separable in $G$, there exists a finite index subgroup $G'<G$ such that $H<G'$ and $\gamma\notin G'$. Then $\Gamma'=G'\cap \Gamma$ is a finite index subgroup of $\Gamma$, with $H<G'\cap \Gamma=\Gamma'$ and $\gamma \notin G'\cap \Gamma=\Gamma'$. So $H$ is also separable in $\Gamma$.

If $\Gamma$ is not LERF, it contains a finitely generated subgroup $H$ which is not separable in $\Gamma$. Then the previous paragraph implies that $H$ is not separable in $G$. So $G$ is not LERF.
\end{proof}

In this paper, the main method to prove a group $G$ is not LERF is to find a descending tower of subgroups of $G$, until we get a subgroup which has a nice structure such that a topological argument can be applied to prove its nonLERFness.

\subsection{Geometric decomposition of irreducible $3$-manifolds}

In this paper, we assume all manifolds are connected and oriented, and all $3$-manifolds are compact and have empty or tori boundary. For any noncompact finite volume hyperbolic manifold $M$, we truncate $M$ by deleting a horocusp for each cusp end of $M$. Then we can consider $M$ as a compact $3$-manifold with tori boundary, and the boundary has an induced Euclidean structure.

Let $M$ be an irreducible 3-manifold with empty or tori boundary. By the geometrization of $3$-manifolds, which is achieved by Perelman and Thurston, exactly one of the following hold.
\begin{itemize}
\item $M$ is geometric, i.e. $M$ supports one of the following eight geometries: $\mathbb{E}^3$, $\mathbb{S}^3$, $\mathbb{S}^2\times \mathbb{E}^1$, $\mathbb{H}^2\times\mathbb{E}^1$, ${\rm Nil}$, ${\rm Sol}$, $\widetilde{{\rm PSL}_2(\mathbb{R})}$, $\mathbb{H}^3$.
\item There is a nonempty minimal union $\mathcal{T}_M\subset M$ of disjoint essential tori and Klein bottles, unique up to isotopy, such that each component of $M\setminus \mathcal{T}_M$ is either Seifert fibered or atoroidal. In the Seifert fibered case, the interior supports both the $\mathbb{H}^2\times \mathbb{E}^1$-geometry and the $\widetilde{{\rm PSL}_2(\mathbb{R})}$-geometry; in the atoroidal case, the interior supports the $\mathbb{H}^3$-geometry.
\end{itemize}

If $M$ has nontrivial geometric decomposition as in the second case, we say that $M$ is a {\it non-geometric $3$-manifold}, and call components of $M\setminus \mathcal{T}_M$ {\it Seifert pieces} or {\it hyperbolic pieces}, according to their geometry. If the components of $M\setminus \mathcal{T}_M$ are all Seifert pieces, $M$ is called a {\it graph manifold}. Otherwise, $M$ contains a hyperbolic piece, and it is called a {\it mixed manifold}. Since we only consider virtual properties of $3$-manifolds in this paper, we can pass to a double cover and assume all components of $\mathcal{T}_M$ are tori.

The geometric decomposition is very closely related to, but slightly different from the more traditional JSJ decomposition. Since these two decompositions agree with each other on some finite cover of $M$, and we are studying virtual properties, we will not make much difference between them.

\subsection{Fibered structures of $3$-manifolds}

In the construction of nonseparable subgroups in Theorem \ref{dim3} and Theorem \ref{singular}, all subgroups have graph of group structures, and the vertex groups are fibered surface subgroups in $3$-manifold groups. So we briefly review the theory of Thurston norm and its relation with fibered structures on $3$-manifolds.

If a $3$-manifold $M$ has a surface bundle over circle structure with $b_1(M)>1$, then $M$ has infinitely many different surface bundle structures. (This works for all dimensions.) These fibered structures of the $3$-manifold $M$ are organized by the Thurston norm on $H_2(M,\partial M; \mathbb{R})$ ($\cong H^1(M;\mathbb{R})$ by duality) defined in \cite{Th2}.

For any $\alpha \in H_2(M,\partial M; \mathbb{Z})$, its Thurston norm is defined by: $$\| \alpha \|=\inf{\{|\chi(T_0)|\ |\ (T,\partial T)\subset(M, \partial M)\ \text{represents}\ \alpha\}},$$ where $T_0\subset T$ excludes $S^2$ and $D^2$ components of $T$. In \cite{Th2}, it is shown that the norm can be extended to $H_2(M,\partial M;\mathbb{R})$ homogeneously and continuously, and the Thurston norm unit ball is a polyhedron with faces dual with elements in $H_1(M;\mathbb{Z})/Tor$. For a general $3$-manifold, the Thurston norm is only a semi-norm, but it is a genuine norm for finite volume hyperbolic $3$-manifolds.

For a top dimensional open face $F$ of the Thurston norm unit ball, let $C$ be the open cone over $F$. In \cite{Th2}, Thurston showed that an integer point $\alpha\in H_2(M,\partial M;\mathbb{R})$ corresponds to a surface bundle structure of $M$ if and only if $\alpha$ is contained in an open cone $C$ as above and all integer points in $C$ correspond to surface bundle structures of $M$. In this case, $C$ is called a {\it fibered cone}, and the corresponding face $F$ is called a {\it fibered face}. For any point (possibly not an integer point) in a fibered cone, we call it a {\it fibered class}.

Thurston's theorem implies that the set of fibered classes of $M$ is an open subset of $H_2(M,\partial M;\mathbb{R})$. In particular, for any fibered class $\alpha\in H_2(M,\partial M;\mathbb{R})$ and any $\beta \in H_2(M,\partial M;\mathbb{R})$, there exists $\epsilon >0$, such that $\alpha+c\beta\in H_2(M,\partial M;\mathbb{R})$ is a fibered class for any $c\in(-\epsilon,\epsilon)$.

\subsection{Virtual retractions of hyperbolic $3$-manifold groups}

In the proof of Theorem \ref{dim3} and \ref{singular}, we need to perturb a fibered class $\alpha\in H_2(M,\partial M;\mathbb{R})$ to get a new fibered class with some desired property. To make sure the desired perturbation exists, we need the virtual retract property of subgroups of hyperbolic $3$-manifold groups.

\begin{definition}
For a group $G$ and a subgroup $H<G$, we say that $H$ is a {\it virtual retraction} of $G$ if there exists a finite index subgroup $G'<G$ containing $H$ and a homomorphism $\phi:G'\to H$, such that $H<G'$ and $\phi|_H=id_H$.
\end{definition}

For a finite volume hyperbolic $3$-manifold $M$, the following dichotomy for a finitely generated infinite index subgroup $H<\pi_1(M)$ holds.
\begin{enumerate}
\item $H$ is a geometrically finite subgroup of $\pi_1(M)$, from the Kleinian group point of view. Equivalently, $H$ is (relatively) quasiconvex in the (relative) hyperbolic group $\pi_1(M)$, from the geometric group theory point of view.
\item $H$ is a geometrically infinite subgroup of $\pi_1(M)$. In this case, $H$ is a virtual fibered surface subgroup of $M$.
\end{enumerate}
Here we do not give the definition of geometrically finite and geometrically infinite subgroups. Readers only need to know that if $H$ is not a virtual fibered surface subgroup, then it is a geometrically finite subgroup. An introduction of geometrically finite subgroups can be found in \cite{Bo} and \cite{Ma} Chapter VI. The proof of the above dichotomy relies on the covering theorem (\cite{Th1}, \cite{Ca}) and the Tameness theorem (\cite{Ag1}, \cite{CG}) on open hyperbolic $3$-manifolds.

In \cite{CDW}, it is shown that (relatively) quasiconvex subgroups of virtually compact special (relative) hyperbolic groups are virtual retractions. The celebrated virtual compact special theorem of Wise (\cite{Wi} for cusped case) and Agol (\cite{Ag3} for closed case) implies that groups of finite volume hyperbolic $3$-manifolds are virtually compact special. These two results together give us the following theorem.

\begin{theorem}\label{virtualretract}
Let $M$ be a finite volume hyperbolic $3$-manifold, $H<\pi_1(M)$ be a geometrically finite subgroup (i.e. $H$ is not a virtual fibered surface subgroup), then $H$ is a virtual retraction of $\pi_1(M)$.
\end{theorem}

\subsection{Arithmetic hyperbolic manifolds}

In this subsection, we briefly review the definition of (standard) arithmetic hyperbolic manifolds. Most material can be found in Chapter 6 of \cite{VS}.

Recall that the hyperboloid model of $\mathbb{H}^n$ is given as the following. Equip $\mathbb{R}^{n+1}$ with a bilinear form $B:\mathbb{R}^{n+1}\times \mathbb{R}^{n+1}\to \mathbb{R}$ defined by $$B\big((x_1,\cdots,x_n,x_{n+1}),(y_1,\cdots,y_n,y_{n+1})\big)=x_1y_1+\cdots+x_ny_n-x_{n+1}y_{n+1}.$$ Then the hyperbolic space $\mathbb{H}^n$ is identified with $$I^n=\{\vec{x}=(x_1,\cdots,x_n,x_{n+1})\ |\ B(\vec{x},\vec{x})=-1,x_{n+1}>0\}.$$ The hyperbolic metric is given by the restriction of $B(\cdot,\cdot)$ on the tangent space of $I^n$.

The isometry group of $\mathbb{H}^n$ consists of all linear transformations of $\mathbb{R}^{n+1}$ that preserve $B(\cdot,\cdot)$ and fix $I^n$. Let $J=\text{diag}(1,\cdots,1,-1)$ be the $(n+1)\times (n+1)$ matrix defining the bilinear form $B(\cdot,\cdot)$, then the isometry group of $\mathbb{H}^n$ is given by $$\text{Isom}(\mathbb{H}^n)\cong PO(n,1;\mathbb{R})=\{X\in GL(n+1,\mathbb{R})\ |\ X^tJX=J\}/(X\sim -X).$$ The orientation preserving isometry group of $\mathbb{H}^n$ is given by $$\text{Isom}_+(\mathbb{H}^n)\cong SO_0(n,1;\mathbb{R}),$$ which is the component of $$SO(n,1;\mathbb{R})=\{X\in SL(n+1,\mathbb{R})\ |\ X^tJX=J\}$$ containing the identity matrix.

\vspace{3mm}

Now we give the definition of {\it standard arithmetic hyperbolic manifolds}, and they are also called {\it arithmetic hyperbolic manifolds of simplest type}.

Let $K\subset \mathbb{R}$ be a totally real number field, and $\sigma_1=id,\sigma_2,\cdots,\sigma_k$ be all the embeddings of $K$ into $\mathbb{R}$. Let $$f(x)=\sum_{i,j=1}^{n+1}a_{ij}x_ix_j,\ a_{ij}=a_{ji}\in K$$ be a nondegenerate quadratic form defined over $K$ with negative inertia index $1$ (as a quadratic form over $\mathbb{R}$). We further suppose that for any $l>1$, the quadratic form $$f^{\sigma_l}(x)=\sum_{i,j=1}^{n+1}\sigma_l(a_{ij})x_ix_j$$ is positive definite, then the information of $K$ and $f$ can be used to define an arithmetic hyperbolic group.

Let $\mathcal{O}_K$ be the ring of algebraic integers in $K$, and $A$ be the $(n+1)\times (n+1)$ matrix defining $f$. Since the negative inertia index of $A$ is $1$, the {\it special orthogonal group of $f$}: $$SO(f;\mathbb{R})=\{X\in SL(n+1,\mathbb{R})\ |\ X^tAX=A\}$$ is conjugate to $SO(n,1;\mathbb{R})$ by a matrix $P$ (satisfying $P^tAP=J$). $SO(f;\mathbb{R})$ has two components, and let $SO_0(f;\mathbb{R})$ be the component containing the identity matrix.

Then we form the set of algebraic integer points $$SO(f;\mathcal{O}_K)=\{X\in SL(n+1,\mathcal{O}_K)\ |\ X^tAX=A\}\subset SO(f;\mathbb{R})$$ in $SO(f;\mathbb{R})$. The theory of arithmetic groups implies that $$SO_0(f;\mathcal{O}_K)=SO(f;\mathcal{O}_K)\cap SO_0(f;\mathbb{R})$$ is conjugate to a lattice of $\text{Isom}_+(\mathbb{H}^n)$ (by the matrix $P$), i.e. it has finite covolume. For simplicity, we abuse notation and still use $SO_0(f;\mathcal{O}_K)$ to denote its $P$-conjugation in $SO_0(n,1;\mathbb{R})\cong\text{Isom}_+(\mathbb{H}^n)$.

Here $SO_0(f;\mathcal{O}_K)\subset \text{Isom}_+(\mathbb{H}^n)$ is called {\it the arithmetic group} defined by number field $K$ and quadratic form $f$, and $\mathbb{H}^n/SO_0(f;\mathcal{O}_K)$ is a finite volume hyperbolic arithmetic orbifold. A hyperbolic $n$-manifold (orbifold) $M$ is called {\it a standard arithmetic hyperbolic manifold (orbifold)} if $M$ is commensurable with $\mathbb{H}^n/SO_0(f;\mathcal{O}_K)$ for some $K$ and $f$.

The arithmetic orbifold $\mathbb{H}^n/SO_0(f;\mathcal{O}_K)$ is noncompact if and only if $f(\vec{x})=0$ has a nontrivial solution $\vec{x}\in K^{n+1}$, which happens only if $K=\mathbb{Q}$ (i.e. $\mathcal{O}_K=\mathbb{Z}$). When $n\geq 4$, $\mathbb{H}^n/SO_0(f;\mathcal{O}_K)$ is noncompact if and only if $K=\mathbb{Q}$.

For this paper, the most important property of standard arithmetic hyperbolic manifolds is that they contain a lot of finite volume hyperbolic $3$-manifolds as totally geodesic submanifolds. This can be done by diagonalizing the matrix $A$ and taking an indefinite $4\times 4$ submatrix.

\vspace{3mm}

The above recipe using quadratic forms over number fields gives all even dimensional arithmetic hyperbolic manifolds (orbifolds). In any odd dimension, there is another family of arithmetic hyperbolic manifolds (orbifolds), which are defined by (skew-Hermitian) quadratic forms over quaternion algebras. We do not give the definition of this family here, and the readers can find a detailed definition in \cite{LM}.

This family of arithmetic hyperbolic manifolds defined by quaternions also have a lot of finite volume hyperbolic $3$-manifolds as totally geodesic submanifolds. This can be done by diagonalizing the quadratic form over quaternions and taking a $2\times 2$ submatrix. Note that this fact is also used in \cite{Ka}.

\vspace{3mm}

In dimension $7$, there is the third way to construct arithmetic hyperbolic manifolds by using octonions. These are sporadic examples, and the author does not know whether these manifolds have totally geodesic (or $\pi_1$-injective) $3$-dimensional submanifolds. All examples in this family are compact manifolds.

\section{NonLERFness of non-geometric $3$-manifold groups}\label{sectiondim3}

In this section, we prove that groups of non-geometric $3$-manifolds are not LERF. The construction of nonseparable (surface) subgroups is enlightened by the construction in \cite{Li1} (and also \cite{RW}). The proof of nonseparability is essentially a computation of the spirality character defined in \cite{Li1}. Here we modify the construction in \cite{Li1}, and give an elementary proof of nonseparability without using the spirality character explicitly.

\subsection{Finite semicovers of non-geometric $3$-manifolds}

We first review the notion of finite semicovers of nongeometric $3$-manifolds, which was introduced in \cite{PW2}.
\begin{definition}
Let $M$ be a nongeometric $3$-manifold with tori or empty boundary. A {\it finite semicover} of $M$ is a compact $3$-manifold $N$ and a local embedding $f:N\to M$, such that its restriction on each boundary component of $N$ is a finite cover to a decomposition torus or a boundary component of $M$.
\end{definition}

For a finite semicover $f:N\to M$, the decomposition tori of $N$ is exactly $f^{-1}(\mathcal{T}_M)\setminus \partial N$, and the restriction of $f$ on each geometric piece of $N$ is a finite cover to the corresponding geometric piece of $M$.

One important property of finite semicovers is given by the following lemma in \cite{Li1}.
\begin{lemma}\label{semicover}
If $N$ is a connected finite semicover of a nongeometric $3$-manifold $M$ with empty or tori boundary, then $N$ has an embedded lifting in a finite cover of $M$. In fact, the semi-covering map $N\to M$ is $\pi_1$-injective and $\pi_1(N)$ is separable in $\pi_1(M)$.
\end{lemma}

\begin{remark}\label{withboundary}
In \cite{Li1}, this lemma is only stated in the case that $M$ is a closed orientable irreducible nongeometric $3$-manifold, but it clearly also holds for irreducible nongeometric $3$-manifolds with nonempty boundary. This is because that we can first take the double $D(M)$ of $M$, apply the closed manifold version of Lemma \ref{semicover} to $N\to D(M)$, then apply Lemma \ref{subgroup} to get separability of $\pi_1(N)$ in $\pi_1(M)$.
\end{remark}

\subsection{Reduction to non-geometric $3$-manifolds with very simple dual graph}

To prove Theorem \ref{dim3}, we reduce to the case that the dual graph of $M$ consists of two vertices and two edges, and $M$ has at least one hyperbolic piece.

Let $M$ be an orientable irreducible nongeometric $3$-manifold with tori or empty boundary. It is known that all graph manifolds have nonLERF fundamental groups (\cite{NW2}), so we can assume that $M$ has at least one hyperbolic piece, i.e. $M$ is a mixed $3$-manifold.

The dual graph of $M$ is a graph with vertices corresponding to geometric pieces of $M$ and edges corresponding to decomposition tori. The following lemma is the first step of our reduction of $3$-manifolds, which reduces the nonLERFness of mixed $3$-manifold groups to a very simple case: the dual graph of $M$ has only two vertices and one edge.

\begin{lemma}\label{2V1E}
Let $M$ be a mixed $3$-manifold, then there exists a $3$-manifold $N=N_1\cup_T N_2$ such that the following hold.
\begin{enumerate}
\item $N_1$ is a cusped hyperbolic $3$-manifold, and $N_2$ is a geometric $3$-manifold.
\item $N_1\cap N_2=T$ is a single torus, and $N=N_1\cup_T N_2$ is a fibered $3$-manifold.
\item $N$ is a finite semicover of $M$, and $\pi_1(N)$ is a subgroup of $\pi_1(M)$.
\end{enumerate}
\end{lemma}

\begin{proof}
By \cite{PW1}, we take a finite cover of $M$ such that it is a fibered $3$-manifold, and still denote it by $M$.

We first suppose that $M$ has at least two geometric pieces. Take any hyperbolic piece $N_1$, and take another (distinct) geometric piece $N_2$ adjacent to $N_1$. It is possible that $N_1\cap N_2$ consists of more than one tori, so let $T$ be one of them. We cut $M$ along all decomposition tori in $\mathcal{T}_M$ except $T$, then the component containing $N_1$ and $N_2$ is the desired $N$, which is clearly a finite semicover of $M$.

The fibered structure on $M$ induces a fibered structure on $N$, since fibered structures of $3$-manifolds are compatible with geometric decomposition. It is easy to see all other desired conditions hold for $N$.

It remains to consider the case that $M$ has only one geometric piece, and we denote it by $N_1$. Since the geometric decomposition of $M$ is nontrivial, there is a decomposition torus $T$ of $M$ that is adjacent to $N_1$ on both sides. Then we take a double cover of $M$ along $T$, and reduce it to the previous case.
\end{proof}

By Lemma \ref{subgroup}, to prove nonLERFness of mixed $3$-manifold groups, we only need to consider the case $M=M_1\cup_T M_2$ as in Lemma \ref{2V1E} (we use $M$ and $M_i$ instead of $N$ and $N_i$ since we will do further constructions). The dual graph of $M$ has two vertices and one edge, which is not our desired model for constructing nonseparable subgroups. Actually, we need a cycle in the dual graph of the $3$-manifold. So we use the following lemma to pass it to a further finite semicover, such that its dual graph consists of two vertices and two edges connecting these two vertices.

\begin{lemma}\label{2V2E}
Let $M=M_1\cup_T M_2$ be a $3$-manifold satisfying the conclusion of Lemma \ref{2V1E}, then there exists a $3$-manifold $N=N_1\cup_{T\cup T'} N_2$ with nonempty boundary such that the following hold.
\begin{enumerate}
\item $N_1$ is a cusped hyperbolic $3$-manifold, and $N_2$ is a geometric $3$-manifold.
\item $N_1\cap N_2=T\cup T'$ is a union of two tori, and $N=N_1\cup_{T\cup T'} N_2$ is a fibered $3$-manifold.
\item The homomorphism $H_1(T\cup T';\mathbb{Z})\to H_1(N_1;\mathbb{Z})$ induced by inclusion is injective.
\item $N$ is a finite semicover of $M$, and $\pi_1(N)$ is a subgroup of $\pi_1(M)$.
\item There exists a fibered surface $S$ in $N$, which is a union of two subsurfaces $S=S_1\cup_{c\cup c'} S_2$, such that $S_i=S\cap N_i$, $c=S\cap T$ and $c'=S\cap T'$. Moreover, $S$ and $S'$ are connected, while both $c$ and $c'$ consist of one circle.
\end{enumerate}
\end{lemma}

\begin{proof}
{\bf Claim.} There exists a $3$-manifold $N=N_1\cup_{T\cup T'}N_2$ satisfying conditions (1)-(4).

We first give the proof of this claim.

\bigskip

We take a base point of $M_1$ on $T$. For $\mathbb{Z}^2\cong\pi_1(T)<\pi_1(M_1)<\text{Isom}_+(\mathbb{H}^3)$, we take any $g\in \pi_1(M_1)$ which maps the fixed point of $\pi_1(T)$ on $S^2_{\infty}$ to a different point. By the Klein combination theorem (Section VII Theorem A.13 of \cite{Ma}), for large enough integer $k$, the subgroup of $\pi_1(M_1)$ generated by $\pi_1(T)$ and $g^k\pi_1(T)g^{-k}$ is isomorphic to the free product of these two groups, i.e. isomorphic to $\mathbb{Z}^2*\mathbb{Z}^2$, and we denote it by $H$.

Since $H<\pi_1(M_1)$ is not a surface subgroup, it is a geometrically finite subgroup. By Theorem \ref{virtualretract}, we can find a finite cover $N_1$ of $M_1$, such that $H<\pi_1(N_1)$ and there exists a retraction homomorphism $\pi_1(N_1)\to H$. Since hyperbolic $3$-manifolds have LERF fundamental groups, by passing to a further finite cover (still denote it by $N_1$), we can assume that $g^k\notin \pi_1(N_1)$, and $N_1$ has at least three boundary components.

Since $g^k\notin \pi_1(N_1)$, any embedded arc $\gamma$ in $N_1$ (starting from the lifted base point) corresponding to $g^k\in \pi_1(M)$ connects two different boundary components of $N_1$, and we denote them by $T_1$ and $T_1'$. Note that the restriction of covering map $N_1\to M_1$ maps both $T_1$ and $T_1'$ to $T$ by homeomorphisms. Then $H<\pi_1(N_1)$ corresponds to the fundamental group of the union of $T_1$, $T_1'$ and $\gamma$. Since $H=\pi_1(T_1\cup T_1' \cup \gamma)$ is a retraction of $\pi_1(N_1)$, $H_1(T_1\cup T_1'\cup \gamma;\mathbb{Z})\cong H_1(T_1\cup T_1';\mathbb{Z})$ is a retraction of $H_1(N_1;\mathbb{Z})$. So condition (3) holds for $N_1$.

If $M_2$ is a cusped hyperbolic $3$-manifold, by doing the same construction for $M_2$, we get a finite cover $N_2\to M_2$ such that two boundary components $T_2$ and $T_2'$ of $N_2$ are mapped to $T$ by homeomorphisms. By identifying $T_1$ and $T_1'$ with $T_2$ and $T_2'$ respectively, we get a semifinite cover $N=N_1\cup_{T\cup T'}N_2$ of $M$ satisfying conditions (1)-(4). Here we use $T$ to denote the image of $T_1$ and $T_2$, and use $T'$ to denote the image of $T_1'$ and $T_2'$.

If $M_2$ is a Seifert fibered space, before doing the above construction for $M_1$, we first do the following preparation. Since $M$ is a fibered $3$-manifold, we have $M=S\times I/\phi$, where $\phi:S\to S$ is a reducible homeomorphism on a surface $S$. By taking some finite cyclic cover $M'$ of $M$ along $S$, we can assume that $M'$ has two adjacent geometric pieces, such that one of them is a cusped hyperbolic $3$-manifold, and another one is homeomorphic to $\Sigma \times S^1$ with $\chi(\Sigma)<0$.

We take the union of these two adjacent pieces along a common torus and get our new $M=M_1\cup_T M_2$ with $M_2=\Sigma\times S^1$. Then we do the same construction for $M_1$ as above to get a finite cover $N_1$. For $M_2$, let $c$ be the boundary component of $\Sigma$ corresponding to the boundary component $T\subset \partial M_2$. Since $\chi(\Sigma)<0$, there exists a double cover $\Sigma'\to \Sigma$ such that there are two boundary components $c_2,c_2'\subset \partial\Sigma'$ that are mapped to $c$ by homeomorphisms.

Then $N_2=\Sigma'\times S^1$ is a finite cover of $M_2$. Let $T_2$ and $T_2'$ be the boundary components of $N_2$ corresponding to $c_2\times S^1$ and $c_2'\times S^1$ respectively, then they are both mapped to $T$ by homeomorphisms. We paste $N_1$ and $N_2$ together to get the desired finite semicover $N=N_1\cup_{T\cup T'}N_2$.

This finishes the proof of the claim.

\vspace{3mm}

Now $N=N_1\cup_{T\cup T'}N_2$ satisfies conditions (1)-(4), so we need to work on condition (5).

Since $M$ is a fibered $3$-manifold, the semicover $N$ has an induced fibered structure. The corresponding fibered surface $S$ might be more complicated than what we want in condition (5), since $S\cap N_i$, $S\cap T$ and $S\cap T'$ may not be connected.

We write $N$ as $N=S\times I/\phi$. Since $N$ has nontrivial torus decomposition, $\phi:S\to S$ is a reducible self-homeomorphism of $S$ (\cite{Th3}). Let $\mathcal{C}$ be the set of reduction circles such that $\phi|:S\setminus \mathcal{C}\to S\setminus \mathcal{C}$ is either pseudo-Anosov or periodic on each $\phi$-component.

We first suppose that there are two components $S_1$ and $S_2$ of $S\setminus \mathcal{C}$ such that $S_i\subset N_i$, while $S_1\cap S_2$ contains two circles $c$ and $c'$ with $c\subset T$ and $c'\subset T'$. Take a positive integer $k$, such that $\phi^k$ preserves each component of $S\setminus \mathcal{C}$ and each component of $\mathcal{C}$. In this case $N'=\big((S_1\cup_{c\cup c'} S_2)\times I\big)/\phi^k$ is a finite semi-cover of $N$. Let $\bar{N}_i=S_i\times I/\phi^k$, and let $\bar{T}$ and $\bar{T}'$ be the components of $\partial \bar{N}_1$ (also $\partial \bar{N}_2$) containing $c$ and $c'$ respectively. Then it is easy to check that $\bar{N}=\bar{N}_1\cup_{\bar{T}\cup \bar{T}'}\bar{N}_2$ satisfies all desired conditions.

If there are not two components of $S\setminus \mathcal{C}$ satisfying the above condition, we need to modify the fibered surface $S$. The new fibered surface is the Haken sum of $S$ and a multiple of $T_1$, and the detail is as the following.

We take a tubular neighborhood $N(T_1)$ of $T_1$ in $N_1$, and give it a coordinate by $N(T_1)=T_1\times I=(S^1 \times I)\times S^1$ such that $$S\cap N(T_1)=\big(\{a_1,a_2,\cdots,a_k\}\times I\big)\times S^1,$$ with $a_1,\cdots,a_k$ following a cyclic order on $S^1$. The fibered structure on $N(T_1)$ is given by a fibered structure of $S^1\times I$, and then it cross with $S^1$. For any integer $j$, we modify the fibered structure on $N(T_1)$ by modifying the fibered structure on $S^1\times I$, then cross with $S^1$. The new fibered structure on $S^1\times I$ is given by a union of disjoint embedded arcs $I_i\subset S^1\times I$, such that $I_i$ connects $(a_i,0)$ to $(a_{i+j},1)$ (modulo $k$), where $i=1,2,\cdots,k$. This fibered structure on $N(T_1)$ can be pasted with the original fibered structure of $N\setminus N(T_1)$ to get a new fibered structure of $N$.

If we start from one component $S_1\subset S\cap N_1$, take any component $S_2\subset S\cap N_2$ such that $S_1\cap S_2\cap T'\ne \emptyset$. Then $S_1\cap T_1$ and $S_2\cap T_2$ are two families of parallel circles on $T$, but maybe any two circles in these two families are not identified with each other. Then we take the above modification of fibered structure for a proper $j$, such that the new fibered surface satisfies the assumption of the previous case.
\end{proof}

Actually, condition (5) is not really necessary in the proof of Theorem \ref{dim3}, but it will make the immersed $\pi_1$-injective surface constructed in Proposition \ref{surface} in a simple shape.

\subsection{Construction of nonseparable surface subgroups}

In this section, we construct a $\pi_1$-injective properly immersed surface in the $3$-manifold $N=N_1\cup_{T\cup T'}N_2$ constructed in Lemma \ref{2V2E}, then prove this surface subgroup is not separable in $\pi_1(N)$.

The following proposition constructs a $\pi_1$-injective properly immersed subsurface in $N$, which is our candidate of nonseparable surface subgroup. Readers can compare this construction with the construction in Section 8 of \cite{Li1}.

\begin{proposition}\label{surface}
For the $3$-manifold $N=N_1\cup_{T\cup T'}N_2$ and fibered surface $S=S_1\cup_{c\cup c'} S_2$ constructed in Lemma \ref{2V2E}, there exists a connected $\pi_1$-injective properly immersed surface $i:\Sigma \looparrowright N$ such that the following hold.
\begin{enumerate}
\item $\Sigma$ is a union of connected surfaces as $\Sigma=\big(\Sigma_{1,1}\cup \Sigma_{1,2}\big)\cup \big(\cup_{k=1}^{2n} \Sigma_{2,k}\big)$, with $i(\Sigma_{1,j})\subset N_1$ and $i(\Sigma_{2,k})\subset N_2$.
\item The restriction of $i$ on $\Sigma_{1,j}$ and $\Sigma_{2,k}$ are embeddings, and their images are fibered surfaces in $N_1$ and $N_2$ respectively.
\item Each $\Sigma_{2,k}$ is a copy of $S_2$ in $N_2$, so $\Sigma_{2,k}$ intersects with both $T$ and $T'$ along exactly one circle.
\item $\Sigma_{1,1}\cap \Sigma_{2,1}$ consists of two circles $s$ and $s'$, with $i(s)\subset T$ and $i(s')\subset T'$.
\item $\Sigma_{1,1}\cap T$ consists of $A$ parallel copies of $c$, and $\Sigma_{1,1}\cap T'$ consists of $B$ parallel copies of $c'$, with $A\ne B$.
\end{enumerate}
\end{proposition}

\begin{proof}
When we cut $N$ along $T\cup T'$ and cut $S$ along $c\cup c'$, wel use $T_i$ and $T_i'$ to denote the copies of $T$ and $T'$ in $N_i$ respectively, and use $c_i$ and $c_i'$ to denote the copies of $c$ and $c'$ in $S_i$ respectively.

Let $\alpha\in H^1(N;\mathbb{Z})$ be the fibered class dual to $S$, and let $\alpha_1=\alpha|_{N_1}$. Then $\alpha_1|_{T_1}$ is dual to $c_1\subset T_1$, and $\alpha_1|_{T_1'}$ is dual to $c_1'\subset T_1'$.

Since $H_1(T_1\cup T_1';\mathbb{Z})\to H_1(N_1;\mathbb{Z})$ is injective, there exists a direct summand $A<H_1(N_1;\mathbb{Z})$ such that $A\cong \mathbb{Z}^4$ and $H_1(T_1\cup T_1';\mathbb{Z})<A$. Since $\mathbb{Z}^4\cong H_1(T_1\cup T_1';\mathbb{Z})<A\cong \mathbb{Z}^4$ is a finite index subgroup, there exists a homomorphism $\tau:A\to \mathbb{Z}$ such that $\tau|_{H_1(T_1;\mathbb{Z})}$ is equall to $l\alpha_1|_{T_1}$ for some $l\in \mathbb{Z}_+$, and $\tau|_{H_1(T_1';\mathbb{Z})}=0$.

Let $\phi:H_1(N_1;\mathbb{Z})\to A$ be a retraction given by the direct summand, then we get a cohomology class $\beta\in H^1(N_1;\mathbb{Z})$ defined by $\tau\circ \phi:H_1(N_1;\mathbb{Z})\to \mathbb{Z}$. By the construction of $\tau$, $\beta|_{T_1}=\tau\circ\phi|_{T_1}=\tau|_{T_1}=l\alpha_1|_{T_1}$ for some $l\in \mathbb{Z}_+$ and $\beta|_{T_1'}=0$.

Since $\alpha_1$ is a fibered class on $N_1$, for large enough $n\in \mathbb{Z}_+$, $\alpha_{1,1}=n\alpha_1+\beta$ and $\alpha_{1,2}=n\alpha_1-\beta$ are both fibered classes in $H^1(N_1;\mathbb{Z})$. Here we can also assume that $n>l$ and $\text{gcd}(n,l)=1$.

Since $\alpha_{1,1}|_{T_1}$ is dual to $n+l$ copies of $c_1$, $\alpha_{1,1}|_{T_1'}$ is dual to $n$ copies of $c_1'$, and $\text{gcd}(n,l)=1$, $\alpha_{1,1}\in H^1(N_1;\mathbb{Z})$ is a primitive class. Similarly, $\alpha_{1,2}\in H^1(N_1;\mathbb{Z})$ is also primitive.

Let $\Sigma_{1,1} \subset N_1$ be the connected fibered surface dual to $\alpha_{1,1}\in H^1(N_1;\mathbb{Z})$ and $\Sigma_{1,2} \subset N_1$ be the connected fibered surface dual to $\alpha_{1,2}$. Then $\Sigma_{1,1}\cap T_1$ consists of $A=n+l$ copies of $c_1$ (as oriented curves), $\Sigma_{1,1}\cap T_1'$ consists of $B=n$ copies of $c_1'$, $\Sigma_{1,2}\cap T_1$ consists of $n-l$ copies of $c_1$, and $\Sigma_{1,2}\cap T_1'$ consists of $n$ copies of $c_1'$. So $(\Sigma_{1,1}\cup \Sigma_{1,2})\cap T_1$ and $(\Sigma_{1,1}\cup \Sigma_{1,2})\cap T_1'$ consist of $2n$ (oriented) copies of $c_1$ and $c_1'$ respectively.

Note that both $S_2\cap T_2$ and $S_2\cap T_2'$ are exactly one (oriented) copy of $c_2$ and $c_2'$ respectively. We take $2n$ copies of $S_2$ in $N_2$, and denote them by $\Sigma_{2,k}$, with $k=1,2.\cdots,2n$. Then we identify parallel circles in $(\Sigma_{1,1}\cup \Sigma_{1,2})\cap T_1$ with $(\cup_{k=1}^{2n} \Sigma_{2,k})\cap T_2$ on $T=T_1=T_2$, and identify parallel circles in $(\Sigma_{1,1}\cup \Sigma_{1,2})\cap T_1'$ with $(\cup_{k=1}^{2n} \Sigma_{2,k})\cap T_2'$ on $T'=T_1'=T_2'$ to get an immersed surface $\Sigma$. In the identification process, we first identify one circle in $\Sigma_{1,1}\cap T_1$ with the circle in $\Sigma_{2,1}\cap T_2$, and identify one circle in $\Sigma_{1,1}\cap T_1'$ with the circle in $\Sigma_{2,1}\cap T_2'$. Then we identify the remaining circles arbitrarily. There are actually many ways to do the further identification, since we can isotopy $\Sigma_{2,k_0}$ such that its intersection with $T_2$ slides over the other circles $\Sigma_{2,k}\cap T_2$, while the other surfaces in $\{\Sigma_{2,k}\}$ are fixed.

It is easy to see that $i:\Sigma\looparrowright N$ is a properly immersed surface, and it satisfies conditions (1)-(5) in the proposition, by the construction.

Moreover, by condition (3) and (5), there exists some $\Sigma_{2,k_0}$ such that both $\Sigma_{1,1}\cap \Sigma_{2,k_0}$ and $\Sigma_{1,2}\cap \Sigma_{2,k_0}$ are not empty. So $\Sigma_{1,1}$ and $\Sigma_{1,2}$ lie in the same connected component of $\Sigma$. Then $\Sigma$ must be connected, since each $\Sigma_{2,k}$ intersects with at least one of $\Sigma_{1,1}$ and $\Sigma_{1,2}$.

\vspace{3mm}

Now we show that $i$ is $\pi_1$-injective by using classical $3$-manifold topology. Suppose there is a map $j:S^1\to \Sigma$ which is not null-homotopic in $\Sigma$, but $i\circ j:S^1\to N$ is null-homotopic in $N$.

We can assume that $i\circ j$ is transverse to the decomposition tori $T\cup T'$, and $j$ minimizes the number of points in $(i\circ j)^{-1}(T\cup T')\subset S^1$, in the homotopy class of $j$. This number is not zero, otherwise, it contradicts with the $\pi_1$-injectivity of fibered surfaces.

Since $i\circ j$ is null-homotopic, it can be extended to a map $k:D^2\to N$ such that $k|_{S^1}=i\circ j$. We can homotopy $k$ relative to $S^1$  such that it is transverse to $T\cup T'$, and $k^{-1}(T\cup T')$ consists of simple arcs in $D^2$.

Then there exists an arc $\alpha\subset S^1$ and an arc component $\beta$ in $k^{-1}(T\cup T')\subset D^2$, such that $\alpha$ and $\beta$ share end points and there are no other components of $k^{-1}(T\cup T')$ lying in the subdisc $B\subset D^2$ bounded by $\alpha\cup \beta$. Without loss of generality, we suppose that $j(\alpha)$ lies in $\Sigma_{1,1}\subset N_1$, $k(\beta)\subset T$, and $K(B)\subset N_1$. Then it is easy to see that the $k$-images of two end points of $\alpha$ lie in the same component of $\Sigma_{1,1}\cap T$, by considering the algebraic intersection number between $\Sigma_{1,1}$ and $\alpha\cup \beta$. Moreover, $k|{\beta}:\beta\to T$ is homotopic to a map into $\Sigma_{1,1}\cap T$, relative to the boundary of $\beta$.

Then it is routine to check that $j|_{\alpha}:\alpha\to \Sigma_{1,1}$ is homotopy to a map with image in $i^{-1}(T)$, relative to the boundary of $\alpha$. After a further homotopy of $j$ supporting on a neighborhood of $\alpha$, we get another $j':S^1\to \Sigma$ which is homotopy to $j$ and has fewer number of points in $(i\circ j')^{-1}(T\cup T')\subset S^1$.

So we get a contradiction with the minimality of $j$, and $i:\Sigma\looparrowright N$ is $\pi_1$-injective.
\end{proof}

The following proposition proves the nonseparability of $i_*(\pi_1(\Sigma))<\pi_1(N)$ constructed in Proposition \ref{surface}. The proof is essentially checking that the spirality character defined in \cite{Li1} for $\Sigma\to N$  is nontrivial, but we do not use the terminology of spirality character, since the picture is relatively simple.

\begin{proposition}\label{nonseparable}
For the properly immersed surface $i:\Sigma\looparrowright N$ constructed in Proposition \ref{surface}, $i_*(\pi_1(\Sigma))<\pi_1(N)$ is a nonseparable subgroup.
\end{proposition}

\begin{proof}
Suppose that $i_*(\pi_1(\Sigma))<\pi_1(N)$ is separable, we will get a contradiction.

Let $\tilde{N}$ be the covering space of $N$ corresponding to $i_*(\pi_1(\Sigma))$. Since each component of $\Sigma\cap i^{-1}(N_k)$ is a fibered surface in $N_k$ for $k=1,2$, it is easy to see that $\tilde{N}$ is homeomorphic to $\Sigma \times \mathbb{R}$. So $i:\Sigma\looparrowright N$ lifts to an embedding $\Sigma \hookrightarrow \tilde{N}$.

Since $i_*(\pi_1(\Sigma))<\pi_1(N)$ is separable, by \cite{Sc}, there exists an intermediate finite cover $\hat{N}\to N$ of $\tilde{N}\to N$ such that $i:\Sigma\looparrowright N$ lifts to an embedding $\hat{i}:\Sigma\hookrightarrow \hat{N}$.

Since $i:\Sigma\looparrowright N$ is a proper immersion, $\hat{i}:\Sigma\hookrightarrow \hat{N}$ is a proper embedding. So $\Sigma$ defines a nontrivial cohomology class $\sigma\in H^1(\hat{N};\mathbb{Z})$.

For each decomposition torus $\hat{T}_s\subset \hat{N}$, suppose $\Sigma\cap \hat{T}_s$ consists of $k_s$ parallel circles. Let $K$ be the least common multiple of all $k_s$. By taking the $K$-sheet cyclic cover of $\hat{N}$ along $\Sigma$ (corresponding to the kernel of $H_1(\hat{N};\mathbb{Z})\xrightarrow{\sigma}\mathbb{Z}\to \mathbb{Z}_K$), we get a further finite cover $\bar{N}\to N$. Then $\Sigma$ embeds into $\bar{N}$, and it intersects with each decomposition torus of $\bar{N}$ exactly once.

Let $\bar{N}_1$ and $\bar{N}_2$ be the geometric pieces of $\bar{N}$ containing $\Sigma_{1,1}$ and $\Sigma_{2,1}$ respectively. Since $\Sigma_{1,1}\cap \Sigma_{2,1}=s\cup s'$, let $\bar{T}$ and $\bar{T}'$ be the decomposition tori in $\bar{N}_1\cap \bar{N}_2$ containing $s$ and $s'$ respectively. Then the finite cover $\bar{N}\to N$ induces finite covers:
\begin{equation*}
\begin{array}{cc}
\bar{N}_1\to N_1,\ \ \bar{N}_2\to N_2,\\
\bar{T}\to T,\ \ \bar{T}'\to T'.
\end{array}
\end{equation*}

Since $\bar{T}\to T$ and $\bar{T}'\to T'$ are both induced by $\bar{N}_1\to N_1$ and $\bar{N}_2\to N_2$, we will get two relations between $\text{deg}(\bar{T}\to T)$ and $\text{deg}(\bar{T}'\to T')$ and get a contradiction.

Since $\Sigma_{1,1}$ is an embedded fibered surface in both $\bar{N}_1$ and $N_1$, $\bar{N}_1$ is a finite cyclic cover of $N_1$ along $\Sigma_{1,1}$. Similarly, $\bar{N}_2$ is a finite cyclic cover of $N_2$ along $\Sigma_{2,1}$.

Since $\Sigma_{1,1}\cap T$ consists of $A$ parallel circles and $\Sigma_{1,1}\cap T'$ consists of $B$ parallel circles, while $\Sigma_{1,1}\cap \bar{T}$ and $\Sigma_{1,1}\cap \bar{T}'$ are both only one circle, $\bar{N}_1\to N_1$ is a cyclic cover whose degree is a multiple of $\text{lcm}(A,B)$, and
\begin{equation}\label{relation1}
A\cdot \text{deg}(\bar{T}\to T)=\text{deg}(\bar{N}_1\to N_1)=B\cdot \text{deg}(\bar{T}'\to T').
\end{equation}

We also have that $\Sigma_{2,1}$ is an embedded fibered surface in both $\bar{N}_2$ and $N_2$. Since $\Sigma_{2,1}\cap T$, $\Sigma_{2,1}\cap T'$, $\Sigma_{2,1}\cap \bar{T}$ and $\Sigma_{2,1}\cap \bar{T}'$ are all just one circle, and $\bar{N}_2\to N_2$ is a finite cyclic cover, we have
\begin{equation}\label{relation2}
\text{deg}(\bar{T}\to T)=\text{deg}(\bar{N}_2\to N_2)=\text{deg}(\bar{T}\to T').
\end{equation}

Equations (\ref{relation1}) and (\ref{relation2}) imply that $A=B$, which contradicts with condition (5) in Proposition \ref{surface}. So $i_*(\pi_1(\Sigma))$ is a nonseparable subgroup of $\pi_1(N)$.
\end{proof}

\begin{remark}\label{geovsalg}
From the proof of Proposition \ref{nonseparable}, readers can see that the main ingredient for proving the nonseparability of $\pi_1(\Sigma)$ is the subsurface $\Sigma_{1,1}\cup_{s\cup s'} \Sigma_{2,1}$. However, the author can not prove that $\pi_1(\Sigma_{1,1}\cup_{s\cup s'} \Sigma_{2,1})$ is nonseparable in $\pi_1(N)$ yet, although it seems quite plausible.

In the proof of Proposition \ref{nonseparable}, we do need the properness of the immersed surface $i:\Sigma\looparrowright N$, so that we can do cyclic cover of $\hat{N}$ along $\Sigma$ to get $\bar{N}$, and then get the contradiction. Actually, most of the proof can be translated to purely group theoretical language, except that the author does not know how to interpret "properly immersed surface" algebraically.
\end{remark}

\subsection{Proof of Theorem \ref{dim3}}

Now we are ready to prove Theorem \ref{dim3}.

\begin{proof}
Suppose that $M$ supports one of eight Thurston's geometries. If $M$ supports the $S^3$- or $S^2\times \mathbb{E}^1$-geometry, since the fundamental group is finite or virtually abelian, LERFness trivially holds. If $M$ supports the $\mathbb{E}^3$-, $\text{Nil}$-, $\mathbb{H}^2\times \mathbb{E}^1$- or $\widetilde{\text{PSL}_2(\mathbb{R})}$-geometry, $M$ is a Seifert manifold and LERFness is proved in \cite{Sc}. If $M$ supports the $\text{Sol}$-geometry, $M$ is virtually a torus bundle over circle, and a proof of LERFness can be found in \cite{NW2}. If $M$ is a hyperbolic $3$-manifold, LERFness is shown by the celebrated works of Wise (\cite{Wi} for cusped case) and Agol (\cite{Ag3} for closed case).

\vspace{3mm}

Now we need to show that non-geometric $3$-manifolds have nonLERF fundamental groups. We first suppose that $M$ is a mixed $3$-manifold, i.e. $M$ has a hyperbolic piece.

If $M$ is not a closed manifold, then Lemma \ref{2V1E} and Lemma \ref{2V2E} imply that $M$ has a finite semicover $N=N_1\cup_{T\cup T'} N_2$ satisfying the conditions in Lemma \ref{2V2E}. In particular, $\pi_1(N)$ is a subgroup of $\pi_1(M)$. Then Proposition \ref{surface} constructs a non-closed surface subgroup (free subgroup) $\pi_1(\Sigma)<\pi_1(N)$, and Proposition \ref{nonseparable} shows that $\pi_1(\Sigma)$ is not separable in $\pi_1(N)$. Finally, Lemma \ref{subgroup} implies that $\pi_1(\Sigma)$ is not separable in $\pi_1(M)$, thus $\pi_1(M)$ is not LERF.

\vspace{3mm}

If $M$ is a closed mixed $3$-manifold, then the above proof also shows the existence of a nonseparable free subgroup in $\pi_1(M)$. We need also to construct a nonseparable closed surface subgroup.

Let $N\to M$ be the finite semicover constructed in Lemma \ref{2V2E} (with $\partial N\ne \emptyset$), and $\Sigma\looparrowright N$ be the $\pi_1$-injective properly immersed surface constructed in Proposition \ref{surface}. To make the geometric picture simpler, we use Lemma \ref{semicover} to find a finite cover $M'$ of $M$, such that $N$ lifts to an embedded submanifold of $M'$.

In this case, the induced map $\Sigma\looparrowright M'$ is an immersion, but is not a proper immersion. So we can not use the proof of Proposition \ref{nonseparable} for this $\Sigma$. Now we extend $\Sigma$ to a closed surface $\Sigma'$, with an immersion $j:\Sigma'\looparrowright M'$. Then we can apply the argument in the proof of Proposition \ref{nonseparable} to prove nonseparability of $\pi_1(\Sigma')<\pi_1(M')$.

The construction of $j:\Sigma'\looparrowright M'$ is actually done in Section 8 of \cite{Li1}, so we only give a sketch here.

Let the boundary components of $\Sigma$ be $s_1,\cdots,s_m$, with each $s_i$ lying on a decomposition torus $T_i\subset M'$. By Theorem 4.11 of \cite{DLW}, there exists an essentially immersed surface $R_i\looparrowright M'$, such that $\partial R_i$ consists of two components $b_i$ and $\bar{b_i}$, with $b_i$ and $\bar{b}_i$ mapped to a positive and a negative multiple of $s_i\subset T_i$ respectively, with the same covering degree. Moreover, a neighborhood of $\partial R_i$ in $R_i$ is mapped to the side of $T_i$ other than $N$, and $R_i$ intersects with $\mathcal{T}_{M'}$ minimally.

Then we take some finite cover of $\hat{\Sigma}\to \Sigma$ such that each boundary component of $\hat{\Sigma}$ that is mapped to $s_i$ has mappeing degree $\text{deg}(b_i\to s_i)$, and take another copy of $\hat{\Sigma}$ with opposite orientation. Together with proper number of copies of $R_i,i=1,\cdots,m$, they can be pasted together to get a $\pi_1$-injective immersed closed surface $\Sigma'\looparrowright M'$. A similar argument as in Proposition \ref{nonseparable} can be applied to $\hat{\Sigma}\subset \Sigma'$ to show that $\pi_1(\Sigma')$ is not separable in $\pi_1(M')$, so it is not separable in $\pi_1(M)$.

\vspace{3mm}

If $M$ is a graph manifold, it was already showed in \cite{NW2} that $\pi_1(M)$ is not LERF. So we do not give detailed construction of nonseparable surface subgroups.

The first step of the construction is to show that $M$ has a finite semicover $N=S\times I/\phi$, where $S=S_1\cup_{c\cup c'}S_2$ and $\phi$ is a composition of Dehn twists along $c$ and $c'$. Then we perturb the fibered structures on both $N_1$ and $N_2$ (since Seifert fibered spaces have less flexible fibered structures) to get a $\pi_1$-injective properly immersed subsurface similar to what we get in Proposition \ref{surface}. Then a similar argument as in Proposition \ref{nonseparable} shows that this surface subgroup is not separable. Here we do need to use the fact that two adjacent Seifert pieces in a graph manifold have incompatible regular fibers on their intersection.

However, it seems not easy to construct a nonseparable closed surface subgroup in a closed graph manifold.
\end{proof}

\begin{remark}\label{engulf}

In \cite{RW}, the authors constructed $\pi_1$-injective properly immersed subsurface $\Sigma\looparrowright M$ for some graph manifold $M$. Then \cite{NW1} proved that $\pi_1(\Sigma)$ is not contained in any finite index subgroup of $\pi_1(M)$ (not engulfed). In the proof of \cite{NW1}, only the infinite plane property of surfaces constructed in \cite{RW} is used. Since the surfaces we constructed in the proof of Theorem \ref{dim3} also have the infinite plane property, for any mixed $3$-manifold $M$, we can find a finite cover $M'\to M$ and a $\pi_1$-injective properly immersed subsurface $\Sigma\looparrowright M'$ such that $\pi_1(\Sigma)$ is not contained in any finite index subgroup of $\pi_1(M')$.

In \cite{NW2}, it is shown that all graph manifold groups contain $$L=\langle x,y,r,s\ |\ rxr^{-1}=x, ryr^{-1}=y, sxs^{-1}=x\rangle$$ as a subgroup. Then the nonLERFness of $L$ implies the nonLERFness of all graph manifold groups. It is easy to see that some mixed manifolds, e.g. double of any cusped hyperbolic $3$-manifold, do not contain $L$ as a subgroup in their fundamental groups. So $L$ is not the source of the nonLERFness of these groups.
\end{remark}

Since any free product of LERF groups is still LERF, we have the following direct corollary of Theorem \ref{dim3}.

\begin{corollary}
Let $M$ be a compact orientable $3$-manifold with empty or tori boundary, then $\pi_1(M)$ is LERF if and only if all prime factors of $M$ support one of Thurston's eight geometries.
\end{corollary}

The knot complements in $S^3$ is also a classical family of interesting $3$-manifolds, and each knot is either a torus knot, or a hyperbolic knot, or a satellite knot. We have the following corollary for knot complements.

\begin{corollary}
Let $M$ be the complement of a knot $K\subset S^3$, then $\pi_1(M)$ is LERF if and only if $K$ is either a torus knot or a hyperbolic knot.
\end{corollary}

\section{Union of two hyperbolic $3$-manifolds along a circle}\label{sectionsingular}

In this section, we will give the proof of Theorem \ref{singular}. The proof is very similar to the proof of Theorem \ref{dim3}. For some lemmas and propositions in this section, we will only give a sketch of the proof, and point out necessary modifications of the corresponding proofs in Section \ref{sectiondim3}.

In the proof of nonLERFness of $\pi_1(M_1\cup_{S^1}M_2)$, we actually only use machineary on hyperbolic $3$-manifolds for $M_1$  (the crucial ingredient is the virtual retract property of its geometrically finite subgroups), and do not have much requirement for $M_2$. So we will have some more general results on nonLERFness of $\mathbb{Z}$-amalgamated groups in Section \ref{general}.

\subsection{NonLERFness of $\pi_1(M_1\cup_{S^1}M_2)$ for hyperbolic $3$-manifolds $M_1$ and $M_2$}

Suppose that $M_1$ and $M_2$ are two finite volume hyperbolic $3$-manifolds (possibly with cusps), and $i_k:S^1\to M_k, k=1,2$ be two essential circles. Here we can assume that both $i_k$ are embeddings into $\text{int}(M_k)$, and denote the image of $i_k$ by $\gamma_k$. It is possible that the element in $\pi_1(M_k)$ corresponding to $\gamma_k$ is a parabolic element or a nonprimitive element. However, for most of the time, the readers can think $\gamma_k$ as a simple closed geodesic in $M_k$.

Let $X=M_1\cup_{\gamma}M_2$ be the space obtained by identifying $\gamma_1$ and $\gamma_2$ by a homeomorphism, then we need to show that $\pi_1(X)$ is not separable. For a standard graph of space, the edge space should be $S^1\times I$. Here we directly paste $M_1$ and $M_2$ together along the circles, which makes the picture simpler. We also give orientations on $\gamma_1$ and $\gamma_2$ such that the pasting preserves orientations on these two circles.

For any point in $X$, either it has a neighborhood homeomorphic to $B^3$ (the open unit ball in $\mathbb{R}^3$) or $B^3_+$ (the points in $B^3$ with non-negative $z$-coordinate), or it has a neighborhood homeomorphic to a union of two $B^3$s along $I_z=B^3\cap (z-\text{axis})$, i.e. $B^3\cup_{I_z}B^3$.

We first give a name for spaces locally look like $B^3$, $B^3_+$ or $B^3\cup_{I}B^3$.

\begin{definition}
A compact Hausdorff space $X$ is called a {\it singular $3$-manifold} if for any point $x\in X$, either it has a neighborhood homeomorphic to $B^3$ or $B^3_+$, or it has a neighborhood homeomorphic to $B^3\cup_{I_z} B^3$ with $x\in I_z$. We call points in the first class {\it regular points}, and call points in the second class {\it singular points}.
\end{definition}

We can think a singular $3$-manifold $X$ as a union of finitely many $3$-manifolds along disjoint simple closed curves, and we call each of these $3$-manifolds a $3$-manifold piece of $X$.

In the proof of Theorem \ref{dim3}, the concept of finite semicover played an important role, so we need to define a corresponding concept for singular $3$-manifolds. Here the set of singular points in singular $3$-manifolds correspond to the set of decomposition tori in $3$-manifolds.

\begin{definition}\label{singularsemicover}
Let $Y,Z$ be two singular $3$-manifolds, a map $i:Y\to Z$ is called a {\it singular finite semicover} if for any point $y\in Y$, one of the following case holds.
\begin{enumerate}
\item $i$ maps a neighborhood of $y$ to a neighborhood of $i(y)$ by homeomorphism.
\item $y$ is a regular point and $i(y)$ is a singular point, such that $i$ maps a $B^3$ neighborhood of $y$ to one of the $B^3$ in a $B^3\cup_{I_z} B^3$ neighborhood of $i(y)$ by homeomorphism.
\end{enumerate}
\end{definition}

Under a singular finite semicover, all singular points are mapped to singular points, and all regular points not lying on a finite union of simple closed curves are mapped to regular points. It maps each $3$-manifold piece of $Y$ to a $3$-manifold piece of $Z$ by a finite cover.

It is easy to see that a singular finite semicover $i:Y\to Z$ induces an injective homomorphism on fundamental group. The author also believes that a singular finite semicover gives a separable subgroup $\pi_1(Y)<\pi_1(Z)$, but we do not need this result here.

The following lemma corresponds to Lemma \ref{2V1E}.

\begin{lemma}\label{singular2V1E}
Let $X=M_1\cup_{\gamma}M_2$ be a union of two finite volume hyperbolic $3$-manifolds along an essential circle, there there exists a singular $3$-manifold $Y=N_1\cup_{c}N_2$ such that the following hold.
\begin{enumerate}
\item $Y$ is a union of two hyperbolic $3$-manifolds $N_1$ and $N_2$, where $N_k$ is a finite cover of $M_k$ ($k=1,2$), and the set of singular points is one oriented circle.
\item Each $N_k$ is a fibered $3$-manifold with a fixed fibered surface $S_k$, such that the algebraic intersection number $[S_k]\cap [c]=1$, for both $k=1,2$.
\item $Y$ is a singular finite semicover of $X$, and $\pi_1(Y)$ is a subgroup of $\pi_1(X)$.
\end{enumerate}
\end{lemma}

\begin{proof}
By Agol's virtual fibering theorem and virtual infinite betti number theorem (\cite{Ag3}), there exists a finite cover $M_1'$ of $M_1$ such that $M_1'$ is a fibered $3$-manifold and $b_1(M_1')>1$. Let $\gamma_1'\subset M_1'$ be one oriented elevation (one component of the preimage) of $\gamma_1\subset M_1$. If $\gamma_1'$ is nullhomologous in $M_1'$, we can use Theorem \ref{virtualretract} to find a further finite cover $M_1''$ such that $\gamma_1'$ lifts to a non null-homologous curve in $M_1''$.

Since the fibered cone is an open set in $H^1(M_1'';\mathbb{R})$, there exists a fibered surface $S_1$ in $M_1''$ which has positive intersection number with $\gamma_1'$. So we have $[S_1]\cap [\gamma_1']=a_1\in \mathbb{Z}_+$, and $\text{deg}(\gamma_1'\to \gamma_1)=b_1$.

By the same construction, we get a finite cover $M_2''\to M_2$ with a fibered surface $S_2$, such that for some oriented elevation $\gamma_2'$ of $\gamma_2$, $[S_2]\cap [\gamma_2']=a_2\in \mathbb{Z}_+$ and $\text{deg}(\gamma_2'\to \gamma_2)=b_2$.

Let $N_1$ be the $a_1b_2$-sheet cyclic cover of $M_1''$ along $S_1$, and $c_1$ be one elevation of $\gamma_1'$. Then $[S_1]\cap [c_1]=1$ and $\text{deg}(c_1\to \gamma_1)=b_1b_2$. Similarly, let $N_2$ be the $a_2b_1$-sheet cyclic cover of $M_2''$ along $S_2$, and $c_2$ be one elevation of $\gamma_1'$. Then $[S_2]\cap [c_2]=1$ and $\text{deg}(c_2\to \gamma_2)=b_1b_2$.

Since $c_1\to \gamma_1$ and $c_2\to \gamma_2$ have the same degree, we can identify $c_1$ and $c_2$ (as oriented curves) to get the desired singular finite semicover $Y=N_1\cup_c N_2$.
\end{proof}

\begin{remark}
Actually, we may get a result as strong as in Lemma \ref{2V1E}, i.e. $Y=N_1\cup_c N_2$ is an $S_1\vee S_2$ bundle over $S^1$. However, we did not state Lemma \ref{singular2V1E} in this way. One reason is that we need to homotopy the curve $c_k$ in $N_k$ to get this fibered structure. Moreover, the closed curve $c_k$ may not (virtually) be a closed orbit of the pseudo-Anosov suspension flow of a (virtual) fibered structure of $N_k$. So it is not a natural object, from the dynamical point of view. Nevertheless, $S_1\vee S_2$ is a homotopy fiber of $Y=N_1\cup_c N_2$, from homotopy point of view.
\end{remark}

For simplicity, we still use $X=M_1\cup_{\gamma}M_2$ to denote the singular $3$-manifold obtained in Lemma \ref{singular2V1E}. Then we have the following lemma corresponding to Lemma \ref{2V2E}.

\begin{lemma}\label{singular2V2E}
For the singular $3$-manifold $X=M_1\cup_{\gamma}M_2$ constructed in Lemma \ref{singular2V1E}, there exists a singular $3$-manifold $Y=N_1\cup_{c\cup c'}N_2$ such that the following hold.
\begin{enumerate}
\item $Y$ is a union of two hyperbolic $3$-manifolds $N_1$ and $N_2$, where each $N_k$ is a finite cover of $M_k$ ($k=1,2$), and the set of singular points consists of two oriented circles.
\item The homomorphism $H_1(c\cup c';\mathbb{Z})\to H_1(N_1;\mathbb{Z})$ induced by inclusion is injective.
\item For each $N_k\ (k=1,2)$, there exists a fibered surface $S_k'\subset N_k$ such that for the algebraic intersection number, $[S_k']\cap [c]=[S_k']\cap [c']=1$ holds.
\item $Y$ is a finite semicover of $X$, and $\pi_1(Y)$ is a subgroup of $\pi_1(X)$.
\end{enumerate}
\end{lemma}

\begin{proof}
Let $\gamma_i$ be the oriented copy of $\gamma$ in $M_i$.

By a similar argument as in the proof of Lemma \ref{2V2E}, and using the virtual retract property of a $\mathbb{Z}*\mathbb{Z}=\langle \pi_1(\gamma),g^n\pi_1(\gamma)g^{-n}\rangle$ subgroup in $\pi_1(M_1)$, we can find a finite cover $N_1$ of $M_1$ and two distinct homeomorphic liftings $c_1$ and $c_1'$ of $\gamma_1\subset M_1$, such that $H_1(c_1\cup c_1';\mathbb{Z})\to H_1(N_1;\mathbb{Z})$ is injective.

In the conclusion of Lemma \ref{singular2V1E}, we fixed a fibered surface $S_1$ in $M_1$ whose algebraic intersection number with $\gamma_1$ is $1$. For an elevated fibered surface $S_1'\subset N_1$, the algebraic intersection numbers of $S_1'$ with $c_1$ and $c_1'$ are both equal to $1$.

By doing a similar construction for $M_2$ (actually a simpler construction works since we do not require condition (2) for $N_2$), we get a finite cover $N_2$ of $M_2$, with two homeomorphic liftings $c_2$ and $c_2'$ of $\gamma_2$, and a fibered surface $S_2'$ of $N_2$ with $[S_2']\cap[c_2]=[S_2']\cap[c_2']=1$.

Then we paste $N_1$ and $N_2$ together by identifying $c_1$ with $c_2$ (denoted by $c$) and identifying $c_1'$ and $c_2'$ (denoted by $c'$) to get the desired singular semicover $Y$.
\end{proof}

For singular $3$-manifolds, we need a definition in this singular world that corresponds to immersed surfaces in $3$-manifolds.

We first define singular surfaces, which plays the same role as surfaces in $3$-manifolds.

\begin{definition}
A compact Hausdorff space $K$ is called a {\it singular surface} if for any point $k\in K$, either it has a neighborhood homeomorphic to $B^2$ or $B^2_+$, or it has a neighborhood homeomorphic to $B^2\vee B^2$, with $k$ lying in the intersection of two discs. We call the points in the first class {\it regular points}, and points in the second class {\it singular points}.
\end{definition}

We can think a singular surface $K$ as a union of finitely many compact surfaces, pasting along finitely many points in the interior. We call each of these surfaces a surface piece of $K$.

Now we define singular immersions from singular surfaces to singular $3$-manifolds.

\begin{definition}\label{singular_immersion}
Let $i:K\to X$ be a map from a singular surface to a singular $3$-manifold. We say that $i$ is a {\it singular immersion} if the following conditions hold.
\begin{enumerate}
\item $i$ maps the singular set of $K$ to the singular set of $X$.
\item The restriction of $i$ on each surface piece of $K$ is a proper immersion from the surface to a $3$-manifold piece of $X$.
\item For any singular point $k\in K$, there exist a $B^2\vee B^2$ neighborhood of $k$ and a $B^3\cup_{I_z} B^3$ neighborhood of $i(k)$, such that $i$ maps two $B^2$s to distinct $B^3$s in $B^3\cup_{I_z} B^3$, and each $B^2$ is mapped to the intersection of $B^3$ with the $xy$-plane by homeomorphism.
\end{enumerate}
\end{definition}

Note that Definition \ref{singular_immersion} is not a good candidate for "proper singular immersion". Under Definition \ref{singular_immersion}, there can be some regular point of $K$ that is mapped to a singular point of $X$. If we consider the corresponding manifold picture, this corresponds to the case that a boundary component of a surface is mapped to a JSJ torus of a $3$-manifold, which is not proper. In the following proposition, we construct a singular immersion that get rid of this picture in the algebraic topology sense, and the readers may compare it with Proposition \ref{surface}.

\begin{proposition}\label{singularsurface}
For the singular $3$-manifold $Y=N_1\cup_{c\cup c'}N_2$ and fibered surfaces $S_k'\subset N_k$ constructed in Lemma \ref{singular2V2E}, there exists a connected singular surface $K$ and a $\pi_1$-injective singular immersion $i:K\looparrowright Y$ such that the following hold.
\begin{enumerate}
\item $K$ is a union of oriented connected subsurfaces as $K=(\Sigma_{1,1}\cup\Sigma_{1,2})\cup (\cup_{k=1}^{2n}\Sigma_{2,k})$, with $i(\Sigma_{1,j})\subset N_1$ and $i(\Sigma_{2,k})\subset N_2$.
\item There are $4n$ singular points in $K$. Each singular point lies in $\Sigma_{1,j}\cap \Sigma_{2,k}$ for some $j\in \{1,2\},k\in\{1,2,\cdots,2n\}$, and each $\Sigma_{2,k}$ contains exactly two singular points.
\item The restriction of $i$ on $\Sigma_{1,j}$s and $\Sigma_{2,k}$s are all embeddings, and their images are fibered surfaces of $N_1$ and $N_2$ respectively.
\item Each $\Sigma_{2,k}$ is a copy of $S_2'$ in $N_2$, and the two singular points in $\Sigma_{2,k}$ are mapped to the intersection of $\Sigma_{2,k}$ with $c$ and $c'$ respectively.
\item $\Sigma_{1,1}\cap \Sigma_{2,1}$ consists of two singular points $p$ and $p'$, with $i(p)\in c$ and $i(p')\in c'$.
\item The algebraic intersection number $[\Sigma_{1,1}]\cap [c_1]=A$ and $[\Sigma_{1,1}]\cap [c_1']=B$, with $A\ne B$; while the algebraic intersection number $[\Sigma_{1,2}]\cap [c_1]=2n-A$ and $[\Sigma_{1,2}]\cap [c_1']=2n-B$. Here $c_1$ and $c_1'$ are the oriented copies of $c$ and $c'$ in $N_1$ respectively.
\item The set $\{\text{singular\ points\ in\ }\Sigma_{1,1}\}\cap  i^{-1}(c)$ has cardinality $A$. Suppose this set is $\{a_1,\cdots,a_A\}$, then $\Sigma_{1,1}$ has positive local intersection number with $c_1$ at each $a_l$. The same statement holds for $\{\text{singular\ points\ in\ }\Sigma_{1,1}\}\cap i^{-1}(c')$ (with $A$ replaced by $B$), $\{\text{singular\ points\ in\ }\Sigma_{1,2}\}\cap i^{-1}(c)$ and $\{\text{singular\ points\ in\ }\Sigma_{1,2}\}\cap i^{-1}(c')$ (with $A$ replaced by $2n-A$ and $2n-B$, respectively).
\item For each $l\in\{1,\cdots, A\}$, take the embedded oriented subarc of $c$ from $i(a_1)$ to $i(a_l)$. Then slightly move it along the positive direction of $c_1$, to get an oriented arc $\rho_l$ with end points away from $\Sigma_{1,1}$. Then the algebraic intersection number between $\Sigma_{1,1}$ and $\rho_l$ is equal to $l-1$. Similar statements also hold for $\{\text{singular\ points\ in\ }\Sigma_{1,1}\}\cap i^{-1}(c')$, $\{\text{singular\ points\ in\ }\Sigma_{1,2}\}\cap i^{-1}(c)$ and $\{\text{singular\ points\ in\ }\Sigma_{1,2}\}\cap i^{-1}(c')$.
\end{enumerate}
\end{proposition}

This proposition looks more complicated than Proposition \ref{surface}, and we give some remarks here.

\begin{remark}
The conditions (1)-(6) in Proposition \ref{singularsurface} correspond to the conditions in Proposition \ref{surface}, and conditions (7) and (8) in Proposition \ref{singularsurface} correspond to the "properness" of this singular immersion. Although we do not assume $$i^{-1}\{\text{singular\ points\ in\ } Y\}=\{\text{singular\ points\ in\ }K\},$$ conditions (6) and (7) imply that the total algebraic intersection number between $\Sigma_{1,j}$ and $c_1$ at the points in $\big(i^{-1}(c)\cap \Sigma_{1,j}\big)\setminus \{\text{singular\ points\ in\ }\Sigma_{1,j}\}$ is zero, and it also holds for $c'$. So it is a weak and algebraic version of $i^{-1}(c\cup c')\cap \Sigma_{1,j}=\{\text{singular\ points\ in\ }\Sigma_{1,j}\}$.

Here we use algebraic intersection number instead of geometric intersection number (or number of components in the intersection) as in Proposition \ref{surface}. For a fibered surface and a closed orbit of the suspension flow (or a boundary component of the $3$-manifold), the algebraic intersection number is always equal to the geometric intersection number (or number of components in the intersection). However, the circles $c_1$ and $c_1'$ in $N_1$ may not be (virtually) closed orbits of the suspension flow, even up to homotopy. Although we can homotopy $c_1$ and $c_1'$ such that their algebraic intersection numbers with one fibered surface are equal to corresponding geometric intersection numbers, but there are two fibered surfaces $\Sigma_{1,1}$ and $\Sigma_{1,2}$ in $N_1$, and we may not be able to do it simultaneously for both $\Sigma_{1,1}$ and $\Sigma_{1,2}$.
\end{remark}

\begin{proof}
By the same argument as in the proof of Proposition \ref{surface}, we can construct two fibered surfaces $\Sigma_{1,1}$ and $\Sigma_{1,2}$ in $N_1$ such that condition (6) holds. Take $2n$ copies of $S_2'$ in $N_2$, and denote them by $\Sigma_{2,1},\cdots,\Sigma_{2,2n}$.

First suppose we choose any $A$, $B$, $2n-A$, $2n-B$ points in $\Sigma_{1,1}\cap c_1$, $\Sigma_{1,1}\cap c_1'$, $\Sigma_{1,2}\cap c_1$, $\Sigma_{1,2}\cap c_1'$ respectively, such that the corresponding surfaces and curves have positive local intersection numbers at these points. If we identify these points with $(\cup_{k=1}^{2n}\Sigma_{2,k})\cap(c_2\cup c_2')$ in an arbitrary way, we get a singular surface $K$ and a singular immersion satisfying conditions (1)-(4), (6) and (7).

So we need to choose these points carefully so that condition (8) holds, and then do the correct pasting such that condition (5) holds.

The choice of these four families of points follows the same process, so we only consider $\Sigma_{1,1}\cap c_1$. Although the algebraic intersection number between $\Sigma_{1,1}$ and $c_1$ is $A$, there might be more geometric intersection points. So we assume that there are $A+2m$ intersection points in $\Sigma_{1,1}\cap c_1$. Take any positive intersection point $a_1'$ in $\Sigma_{1,1}\cap c_1$. By following the orientation of $c_1$, we denote the other points of $\Sigma_{1,1}\cap c_1$ by $a_2',\cdots,a_{A+2m}'$. For any $l\in \{1,\cdots,A+2m\}$, take the embedded oriented subarc in $c_1$ from $a_1'$ to $a_l'$, then move it slightly along the positive direction of $c_1$, and denote it by $\rho_l'$. Whenever we move from $a_l'$ to $a_{l+1}'$, the algebraic intersection number $[\Sigma_{1,1}]\cap [\rho_l']$ differs from $[\Sigma_{1,1}]\cap [\rho_{l+1}']$ by $1$ or $-1$, depending on whether $\Sigma_{1,1}$ intersects with $c_1$ positively or negatively at $a_{l+1}'$. Since $[\Sigma_{1,1}]\cap [\rho_1']=0$ and $[\Sigma_{1,1}]\cap [\rho_{A+2m}']=A-1$, it is easy to find $A$ points in $\{a_1',a_2',\cdots,a_{A+2m}'\}$ (with $a_1=a_1'$), such that they are all positive intersection points and satisfy condition (8).

Then we can paste the $2n$ points in $(\Sigma_{1,1}\cup \Sigma_{1,2})\cap c_1$ (and $(\Sigma_{1,1}\cup \Sigma_{1,2})\cap c_1'$) chosen above with the $2n$ points in $(\cup_{k=1}^{2n}\Sigma_{2,k})\cap c_2$ (and $(\cup_{k=1}^{2n}\Sigma_{2,k})\cap c_2'$), to get a connected singular immersed surface $i:K\looparrowright Y$. By doing isotopy of $\Sigma_{2,1}$ in $N_2$, we can make sure the pasting satisfies condition (5).

The $\pi_1$-injectivity of $i:K\looparrowright Y$ follows from the same $\pi_1$-injectivity argument in Lemma \ref{surface}. Note that we do need condition (8) here.
\end{proof}

Then we can show that the above $\pi_1$-injective singular immersion gives a nonseparable subgroup in $\pi_1(Y)$. This proof is similar to the proof of Proposition \ref{nonseparable}.

\begin{proposition}\label{singularnonseparable}
For the singular immersion $i:K\looparrowright Y$ constructed in Proposition \ref{singularsurface}, $i_*(\pi_1(K))<\pi_1(Y)$ is a nonseparable subgroup.
\end{proposition}

\begin{proof}
Suppose that $i_*(\pi_1(K))<\pi_1(Y)$ is separable, then we will get a contradiction.

Since each surface piece of $K$ is mapped to a fibered surface in the corresponding $3$-manifold piece of $Y$, the covering space $\tilde{Y}$ of $Y$ corresponding to $\pi_1(K)$ is homeomorphic to a union of $\Sigma_{1,j}\times \mathbb{R}$ (with $j=1,2$) and $\Sigma_{2,k}\times \mathbb{R}$ (with $k=1,\cdots,2n$), by pasting along the preimage of $c_i$ and $c_i'$ (with $i=1,2$). In particular, $i:K\looparrowright Y$ lifts to an embedding in $\tilde{Y}$. By the separability of $i_*(\pi_1(K))$, \cite{Sc} implies that there exists an intermediate finite cover $p:\hat{Y}\to Y$ of $\tilde{Y}\to Y$ such that $i:K\looparrowright Y$ lifts to an embedding $\hat{i}:K\hookrightarrow \hat{Y}$.

In Proposition \ref{nonseparable}, we took a finite cyclic cover of $\hat{N}$ along $\Sigma$. It can be done either geometrically, i.e. take finitely many copies of $\hat{N}\setminus \Sigma$ and paste them together, or algebraically, i.e. take a finite cyclic cover dual to the cohomology class defined by $\Sigma$. Here we will follow the algebraic process.

\vspace{3mm}

Now we show that $K\subset \hat{Y}$ defines a cohomology class $\kappa\in H^1(Y;\mathbb{Z})$, by using duality, i.e. taking algebraic intersection number.

Since $K$ intersects with each $3$-manifold piece of $\tilde{Y}$, it also intersects with each $3$-manifold piece of $\hat{Y}$. For each $3$-manifold piece $\hat{N}_s$ of $\hat{Y}$, $K\cap \hat{N}_s$ is a properly embedded oriented surface in $\hat{N}_s$, so it defines a cohomology class $\kappa_s\in H^1(\hat{N}_s;\mathbb{Z})$.

For each component $\hat{c}$ of $p^{-1}(c\cup c')$, suppose that it is adjacent to $\hat{N}_1$ and $\hat{N}_2$, then we need to show that $\kappa_1|_{\hat{c}}=\kappa_2|_{\hat{c}}$, i.e. the algebraic intersection numbers $[\hat{N}_1\cap K]\cap [\hat{c}]$ and $[\hat{N}_2\cap K]\cap [\hat{c}]$ are equal to each other. Here $\hat{N}_1$ and $\hat{N}_2$ are finite covers of $N_1$ and $N_2$ respectively.

Since $\Sigma_{1,1}$ and $\Sigma_{1,2}$ are different fibered surfaces in $N_1$, only one of them lies in $\hat{N}_1$. Without loss of generality, we suppose that $K\cap \hat{N}_1=\Sigma_{1,1}$, and $\hat{c}$ is a component of $p^{-1}(c)$. All other cases follow from the same argument.

Since $\Sigma_{1,1}$ is a fibered surface in both $\hat{N}_1$ and $N_1$, $\hat{N}_1\to N_1$ is a finite cyclic cover dual to $\Sigma_{1,1}$, and let the covering degree be $D$. Recall that $[\Sigma_{1,1}]\cap [c_1]=A$. Then $p^{-1}(c)\cap \hat{N}_1$ has $\text{gcd}(A,D)$ many components ($\hat{c}$ is one of them), and each of them has algebraic intersection number $\frac{A}{\text{gcd}(A,D)}$ with $\Sigma_{1,1}$.

\vspace{3mm}

So we have that $\langle\kappa_1,\hat{c}\rangle=\frac{A}{\text{gcd}(A,D)}$, and need to show $\langle\kappa_2,\hat{c}\rangle=\frac{A}{\text{gcd}(A,D)}$.

We first show that for $\hat{i}|_{\Sigma_{1,1}}:\Sigma_{1,1}\to \hat{N}_1$, there are exactly $\frac{A}{\text{gcd}(A,D)}$ points in $\{a_1,\cdots,a_A\}$ mapped to $\hat{c}$. For two points $a_s,a_t\in \{a_1,\cdots,a_A\}$, if $\hat{i}(a_s)$ and $\hat{i}(a_t)$ lie in the same component of $p^{-1}(c)\cap \hat{N}_1$, there is an oriented subarc $\tau$ of $p^{-1}(c)$ from $\hat{i}(a_s)$ to $\hat{i}(a_t)$. Take an oriented path $\gamma$ in $\Sigma_{1,1}$ form $a_s$ to $a_t$. Then $\tau\cdot \hat{i}(\gamma^{-1})$ is a loop in $\hat{N}_1$ and it projects to a loop $\delta$ in $N_1$.

Since $\tau\cdot\hat{i}(\gamma^{-1})$ is a loop in $\hat{N}_1$, the algebraic intersection number of $\Sigma_{1,1}$ with $\delta$ is a multiple of $D$. On the other hand, $\delta$ consists of the projection of $\tau$ and $\hat{i}(\gamma^{-1})$ in $N_1$, which helps us to compute the algebraic intersection number by another way. Since $\gamma$ lies in $\Sigma_{1,1}$, its projection in $N_1$ has $0$ algebraic intersection number with $\Sigma_{1,1}$. Since $p\circ \tau$ is a path on $c_1$ with initial point $a_s$ and terminal point $a_t$, by condition (8) in Proposition \ref{singularsurface}, the algebraic intersection number of $\Sigma_{1,1}$ with $p\circ \tau$ is $nA+(t-s)$ for some $n\in \mathbb{Z}$. (Actually, we need to slightly push $\tau$ and $\hat{i}(\gamma^{-1})$ along the positive direction of the corresponding component of $p^{-1}(c)$, such that their endpoints are away from $\Sigma_{1,1}$.)

From the two ways of computing the algebraic intersection number between $\Sigma_{1,1}$ and $\delta$, we get that $mD=nA+(t-s)$ holds for some integers $m$ and $n$. So $t-s$ is a multiple of $\text{gcd}(A,D)$. It implies that for each component of $p^{-1}(c)\cap \hat{N}_1$, there are exactly $\frac{A}{\text{gcd}(A,D)}$ points in $\{a_1,\cdots,a_A\}$ mapped to it. In particular, it holds for $\hat{c}$.

There are exactly $\frac{A}{\text{gcd}(A,D)}$ points in $\{a_1,\cdots,a_A\}\cap\hat{c}$. Each of them lies in a fibered surface in $K\cap \hat{N}_2$, and the algebraic intersection number between the fibered surface and $\hat{c}$ is $1$. So we have $\langle\kappa_2, \hat{c}\rangle=\frac{A}{\text{gcd}(A,D)}=\langle \kappa_1, \hat{c}\rangle$.

By an M-V sequence argument, we get that $\hat{i}:K\hookrightarrow \hat{Y}$ defines a cohomology class $\kappa\in H^1(Y;\mathbb{Z})$, by taking the algebraic intersection number of any $1$-cycle in $Y$ with $K$.

\vspace{3mm}

As in Proposition \ref{nonseparable}, we take a finite cover of $\hat{Y}$ dual to $\kappa$ to get a further finite cover $q:\bar{Y}\to Y$ such that $K$ embedds in $\bar{Y}$. We can further require that each component of $q^{-1}(c\cup c')$ intersects with exactly two surface pieces in $K$, with algebraic intersection number $1$. Let $\bar{N}_1$ and $\bar{N}_2$ be the $3$-manifold pieces of $\bar{Y}$ containing $\Sigma_{1,1}$ and $\Sigma_{2,1}$ respectively, and let $\bar{c},\bar{c}'\subset \bar{N}_1\cap \bar{N}_2$ be singular circles containing the two points in $\Sigma_{1,1}\cap \Sigma_{1,2}$. Then we can compute the relation between $\text{deg}(\bar{c}\to c)$ and $\text{deg}(\bar{c}'\to c')$ from $\text{deg}(\bar{N}_1\to N_1)$ and $\text{deg}(\bar{N}_2\to N_2)$ by two ways, and get a contradiction as in Proposition \ref{nonseparable}.
\end{proof}

Now we are ready to prove Theorem \ref{singular}.

\begin{proof}
For a singular $3$-manifold $X=M_1\cup_{\gamma}M_2$, Lemma \ref{singular2V1E} and Lemma \ref{singular2V2E} imply that there exists a singular finite semicover $Y=N_1 \cup_{c\cup c'}N_2$ of $X$ such that the conditions in Lemma \ref{singular2V2E} hold.

By Proposition \ref{singularsurface}, there exists a $\pi_1$-injective singular immersion $i:K\looparrowright Y$ satisfying the conditions in Proposition \ref{singularsurface}. Then Proposition \ref{singularnonseparable} implies that $i_*(\pi_1(K))$ is not separable in $\pi_1(Y)$. Since $\pi_1(Y)$ is a subgroup of $\pi_1(X)$, Lemma \ref{subgroup} implies that $i_*(\pi_1(K))$ is not separable in $\pi_1(X)$.

If both $M_1$ and $M_2$ are $3$-manifolds with boundary, $K$ is a union of surfaces with boundary along finitely many points, so $\pi_1(K)$ is a free group. If at least one of $M_1$ and $M_2$ is a closed $3$-manifold, then $K$ is a union of closed surfaces and (possibly empty set of) bounded surfaces along finitely many points, so $\pi_1(K)$ is a free product of free groups and surface groups.
\end{proof}

The following direct corollary of Theorem \ref{singular} implies that any HNN extension of a hyperbolic $3$-manifold group along cyclic subgroups is not LERF.

The readers may compare this corollary with the result in \cite{Ni2}, which gives a sufficient and necessary condition for an HNN extension of a free group along cyclic subgroups being LERF. Note that Niblo's condition holds for a generic pair of cyclic subgroups in a free group.

\begin{corollary}\label{HNN}
Let $M$ be a finite volume hyperbolic $3$-manifold, and $A,B<\pi_1(M)$ be two infinite cyclic subgroups with an isomorphism $\phi:A\to B$, then the HNN extension $$\pi_1(M)*_{A^t=B}=\langle \pi_1(M),t\ |\ tat^{-1}=\phi(a),\forall a\in A\rangle$$ is not LERF.
\end{corollary}

\begin{proof}
Let $\pi_1(M)_{A^t=B}\to \mathbb{Z}_2$ be the homomorphism which kills all elements in $\pi_1(M)$ and maps $t$ to $\bar{1}\in \mathbb{Z}_2$. Then the kernel $H$ is an index two subgroup of $\pi_1(M)*_{A^t=B}$.

The subgroup $H$ has a graph of group structure such that the graph consists of two vertices and two edges connecting these two vertices. The vertex groups are two copies of $\pi_1(M)$, and the edge groups are both infinite cyclic. So $H$ contains a subgroup which is a $\mathbb{Z}$-amalgamation of two copies of $\pi_1(M)$. Then Theorem \ref{singular} and Lemma \ref{subgroup} imply that $\pi_1(M)_{A^t=B}$ is not LERF.
\end{proof}

\subsection{More general cases}\label{general}

Actually, the proof of Theorem \ref{singular} only uses the machinery on hyperbolic $3$-manifolds for $M_1$, and $M_2$ only need to satisfy some mild conditions. So we have the following generalization of Theorem \ref{singular}.

\begin{theorem}\label{moresingular}
Let $M_1$ be a finite volume hyperbolic $3$-manifold, and $M_2$ be a compact fibered manifold over the circle, i.e. $M_2=N\times I/\phi$ for some orientation preserving self-homeomorphism $\phi:N\to N$ on a compact $n$-manifold $N$ ($n>0$). We also suppose that $\pi_1(N)$ has some nontrivial finite quotient.

Let $S^1\to M_1$ be an essential circle in $M_1$, and $S^1\to M_2$ be an essential circle which has nonzero algebraic intersection number with $N$. Then the $\mathbb{Z}$-amalgamation $$\pi_1(M_1\cup_{S^1}M_2)$$ is not LERF.
\end{theorem}

We give a sketch of the proof which parallels the proof of Theorem \ref{singular}. It is not hard to see that the proof works even if we only assume that $N$ is a CW-complex and $M_2$ is a mapping torus of $N$ (still $\pi_1(N)$ need to have a nontrivial finite quotient), but we prefer to only state the result for the manifold case.

\begin{proof}
At first, we can find a singular finite semicover $M_1'\cup_{\gamma}M_2'$ such that similar conditions in Lemma \ref{singular2V1E} hold. For $M_1$, we still use the virtual fibering theorem, virtual infinite betti number theorem, and the virtual retract property to find a fibered structure in some finite cover of $M_1$, such that conditions (1) and (2) in Lemma \ref{singular2V1E} hold. For $M_2$, it already has a fibered structure, and we may only need to take a finite cyclic cover of $M_2$ along $N$.

Then we can find a further singular semicover $N_1\cup_{c\cup c'}N_2$ such that similar conditions in Lemma \ref{singular2V2E} hold. For $M_1'$, we still use the virtual retract property and LERFness to get a finite cover $N'$ satisfying conditions (2) and (3) in Lemma \ref{singular2V2E}. For $M_2'$, we use the fact that $\pi_1(N)$ admits a nontrivial finite quotient to find a finite cover $N_2$ of $M_2'$, such that the preimage of $\gamma$ in $N_2$ contains at least two components, and they are mapped to $\gamma$ by homeomorphisms.

In the construction of the $\pi_1$-injective immersed singular object (not a singular surface if $n\ne 2$) in Proposition \ref{singularsurface}, we only perturb fibered structures in $N_1$, do all nontrivial works over there, and always use the original fibered structure of $N_2$. So the same construction gives a $\pi_1$-injective immersed singular object in $N_1\cup_{c\cup c'}N_2$, which satisfies the conditions in Proposition \ref{singularsurface}.

The proof of Proposition \ref{singularnonseparable} does not use any $3$-manifold topology. It only uses the fiber bundle over circle structures and counts covering degrees. So the same proof shows that the above $\pi_1$-injective immersed singular object gives a nonseparable subgroup in $\pi_1(N_1\cup_{c\cup c'}N_2)$, which is also nonseparable in $\pi_1(M_1\cup_{S^1}M_2)$.

The nonseparable subgroup constructed above is a free product of surface groups, finite index subgroups of $\pi_1(N)$ and free groups.
\end{proof}

Since the perturbation of fiber bundle over circle structures works in any dimension, we have the following further corollary.

\begin{corollary}
Let $M_1$ be a finite volume hyperbolic $3$-manifold, and $M_2$ be a compact manifold with a fiber bundle over circle structure and $b_1(M_2)\geq 2$.

Let $S^1\to M_1$ be an essential circle in $M_1$, and $S^1\to M_2$ be a circle in $M_2$ with nonzero image in $H_1(M_2;\mathbb{Q})$. Then the $\mathbb{Z}$-amalgamation $$\pi_1(M_1\cup_{S^1}M_2)$$ is not LERF.
\end{corollary}

\begin{proof}
At first, $b_1(M_2)\geq 2$ implies that, for any fiber bundle over circle structure $M_2=N\times I/\phi$, $b_1(N)\geq 1$ holds. So for any such $N$, $\pi_1(N)$ has a nontrivial finite quotient.

We take any fibered structure of $M_2$ and write $M_2$ as $M_2=N\times I/\phi$. By perturbing the fibered structure on $M_2$, we can assume that $[N]$ has non zero algebraic intersection number with $[S^1]\in H_1(M_2;\mathbb{Z})$. So we are in the situation of Theorem \ref{moresingular}, and $\pi_1(M_1\cup_{S^1}M_2)$ is not LERF.
\end{proof}

\section{NonLERFness of arithmetic hyperbolic manifold groups}\label{sectionarithmetic}

In this section, we give the proof of Theorem \ref{arithmetic} and Theorem \ref{noncompact}, and give some further results on nonLERFness of high dimensional nonarithmetic hyperbolic manifold groups. These results imply that most known examples of high dimensional hyperbolic manifolds have nonLERF fundamental groups.

For all proofs in this section, to prove a group is not LERF, we only need to show that it contains a subgroup isomorphic to one of the nonLERF groups in Theorem \ref{dim3} or Theorem \ref{singular}.

We start with proving Theorem \ref{noncompact}, which claims that all noncompact arithmetic hyperbolic manifolds with dimension $\geq 4$ have nonLERF fundamental groups.

\begin{proof}

We first consider the case that $M$ is a noncompact standard arithmetic hyperbolic manifold.

We first show that $M$ contains a (immersed) noncompact totally geodesic $3$-dimensional submanifold $N$. This is well-known for experts, but the author did not find a reference on it, so we give a short proof here.

Since $M$ is noncompact, it is defined by $\mathbb{Q}$ and a nondegenerate quadratic form $f:\mathbb{Q}^{m+1}\to \mathbb{Q}$ with negative inertia index $1$. Let the symmetric bilinear form defining $f$ be denoted by $B(\cdot,\cdot)$.

Since $M$ is not compact, $f$ represents $0$ nontrivially in $\mathbb{Q}^{m+1}$, thus there exists $\vec{w}\ne \vec{0}\in \mathbb{Q}^{m+1}$ such that $B(\vec{w},\vec{w})=f(\vec{w})=0$. Since $f$ is nondegenerate, there exists $\vec{v}\ne 0\in \mathbb{Q}^{m+1}$ such that $B(\vec{v},\vec{w})\ne 0$. Let $V=\text{span}_{\mathbb{Q}}(\vec{v},\vec{w})$. Then it is easy to check that $V^\perp\cap V=\{\vec{0}\}$ and the restriction of $B(\cdot,\cdot)$ on $V^\perp$ is positive definite.

Let $(\vec{v}_1,\cdots,\vec{v}_{m-1})$ be a $\mathbb{Q}$-basis of $V^\perp$ such that $B(\vec{v}_i,\vec{v}_j)=\delta_{ij}$ for $i,j\in\{1,\cdots,m-1\}$. Then $W=\text{span}_{\mathbb{Q}}(\vec{v_1},\vec{v_2},\vec{v},\vec{w})$ is a $4$-dimensional subspace of $\mathbb{Q}^{m+1}$, such that the restriction of $f$ on $W$ has negative inertia index $1$, and $f$ represents $0$ nontrivially in $W$.

So $W$ and $f|_W$ define a (immersed) noncompact totally geodesic $3$-dimensional suborbifold in $\mathbb{H}^{m}/SO_0(f,\mathbb{Z})$, which gives a (immersed) noncompact totally geodesic $3$-dimensional submanifold $N^3$ in $M$.

\vspace{3mm}

Now we are ready to prove the theorem. Here we consider $M$ and $N^3$ as compact manifolds, by truncating their horocusps.

Each boundary component of $M$ has an Euclidean structure, so it is finitely covered by $T^{m-1}$, and each boundary component of $N$ is homeomorphic to $T^2$. We first take two copies of $N$. For each $T^2$ component of $\partial N$, take a long enough immersed $T^2\times I$ in the corresponding boundary component of $M$, which is finitely covered by $T^{m-1}=(T^2\times S^1)\times T^{m-4}$, such that the $T^2$ factor is identified with $T^2\subset \partial N^3$, and the $I$ factor wraps around the $S^1$ factor. This construction is same with the Freedman tubing construction in dimension $3$. In \cite{LR}, it is shown that as long as the $I$ factor wraps around $S^1$ sufficiently many times, this immersed $N\cup (\partial N\times I)\cup N$ is $\pi_1$-injective, so $\pi_1(N\cup (\partial N\times I)\cup N)<\pi_1(M)$.

Topologically, $N\cup (\partial N\times I)\cup N$ is just the double of $N$ along $\partial N$. Since the double of $N$ is a closed mixed $3$-manifold with nontrivial geometric decomposition, Theorem \ref{dim3} implies that $\pi_1(N\cup (\partial N\times I)\cup N)$ is not LERF. Then Lemma \ref{subgroup} implies $\pi_1(M)$ is not LERF.

Moreover, Theorem \ref{dim3} implies the existence of nonseparable free subgroups and nonseparable surface subgroups in $\pi_1(M)$.

\vspace{3mm}

If $M$ is a noncompact arithmetic hyperbolic manifold defined by quaternions, it also contains noncompact $3$-dimensional totally geodesic submanifolds, by doing the same process as above for quadratic forms over quaternions. So the above proof also works in the quaternion case.

Since $7$-dimensional arithmetic hyperbolic manifolds defined by octonions are all compact, the proof is done.

\end{proof}

Then we give the proof of Theorem \ref{arithmetic}, which claims that all arithmetic hyperbolic manifolds with dimension $\geq 5$ which are not those sporadic examples in dimension $7$ have nonLERF fundamental groups. In this proof, we use two totally geodesic $3$-dimensional submanifolds, instead of just using one such submanifold as in the proof of Theorem \ref{noncompact}.

\begin{proof}

We first suppose that $M^m$ is a standard arithmetic hyperbolic manifold, with $m\geq 5$.

By the definition of standard arithmetic hyperbolic manifolds, there exists a totally real number field $K$, and a nondegenerate quadratic form $f:K^{m+1}\to K$ defined over $K$, such that the negative inertial index of $f$ is $1$ and $f^{\sigma}$ is positive definite for all non-identity embeddings $\sigma:K\to \mathbb{R}$. Moreover, $\pi_1(M)$ is commensurable with $SO_0(f;\mathcal{O}_K)$. So to prove $\pi_1(M)$ is not LERF, we need only to show $SO_0(f;\mathcal{O}_K)$ is not LERF.

We first diagonalize the quadratic form $f$ such that the symmetric matrix corresponding to $f$ is $A=\text{diag}(k_1,\cdots,k_m,k_{m+1})$ with $k_1,\cdots,k_m>0$ and $k_{m+1}<0$.

\vspace{3mm}

First suppose that there exists $i\in\{1,\cdots,m\}$ such that $-\frac{k_i}{k_{m+1}}$ is not a square in $K$, and we can assume $i=1$. Then $f$ has two quadratic subforms defined by $\text{diag}(k_1,k_2,k_3,k_{m+1})$ and $\text{diag}(k_1,k_4,k_5,k_{m+1})$ respectively. These two subforms satisfy the conditions for defining arithmetic groups in $\text{Isom}_+(\mathbb{H}^3)$, and we denote these two subforms by $f_1$ and $f_2$.

Then $SO_0(f_1;\mathcal{O}_K)$ and $SO_0(f_2;\mathcal{O}_K)$ are both subgroups of $SO_0(f;\mathcal{O}_K)$. Each of them fix a $3$-dimensional totally geodesic plane in $\mathbb{H}^m$, and these two planes perpendicularly intersect with each other along a $1$-dimensional bi-infinite geodesic. (Here we do use that $m\geq 5$.) We denote these two $3$-dimensional planes by $P_1$ and $P_2$ with $P_1\cap P_2=L$. Then $M_i=P_i/SO_0(f_i;\mathcal{O}_K)$ is a hyperbolic $3$-orbifold for each $i=1,2$. Moreover, it is easy to see that $SO_0(f_1;\mathcal{O}_K)\cap SO_0(f_2;\mathcal{O}_K)=SO_0(f_3;\mathcal{O}_K)$, where $f_3$ is defined by $\text{diag}(k_1,k_{m+1})$ and $SO_0(f_3;\mathcal{O}_K)$ fixes $L$. The condition that $-\frac{k_1}{k_{m+1}}$ is not a square in $K$ implies that $f_3$ only represents $0$ trivially in $K^2$, so $SO_0(f_3;\mathcal{O}_K)\cong \mathbb{Z}$.

By a routine argument in hyperbolic geometry and using LERFness of hyperbolic $3$-manifold groups (e.g. see Lemma 7.1 of \cite{BHW}), there exist torsion free finite index subgroups $\Lambda_i<SO_0(f_i;\mathcal{O}_K)$ with $SO_0(f_3;\mathcal{O}_K)<\Lambda_i$ for $i=1,2$, and the subgroup of $SO_0(f;\mathcal{O}_K)$ generated by $\Lambda_1$ and $\Lambda_2$ is isomorphic to $\Lambda_1*_{\mathbb{Z}}\Lambda_2$.

So $SO_0(f;\mathcal{O}_K)$ contains a subgroup $\Lambda_1*_{\mathbb{Z}}\Lambda_2$, which is the fundamental group of $M_1\cup_{\gamma}M_2$ for two hyperbolic $3$-manifolds $M_1$ and $M_2$. By Theorem \ref{singular}, $SO_0(f;\mathcal{O}_K)$ is not LERF, and $\pi_1(M)$ is not LERF.

If $M$ is closed, then both $M_1$ and $M_2$ are closed, and the nonseparable subgroup can be chosen to be a free product of closed surface groups and free groups. If $M$ has cusps, the nonseparable subgroup might be a free group.

If $-\frac{k_i}{k_{m+1}}$ is a square in $K$ for all $i\in \{1,\cdots,m\}$, then the quadratic form $f$ is equivalent to the diagonal form $\text{diag}(\underbrace{1,\cdots,1}_m,-1)$. It is easy to check that $f$ is also equivalent to the diagonal form $\text{diag}(2,2,\underbrace{1,\cdots,1}_{m-2},-1)$, and we reduce to the previous case.

\vspace{3mm}

If $M^m$ is an arithmetic hyperbolic manifold defined by a quadratic form over quaternions, we can also find two totally geodesic $3$-dimensional submanifolds intersecting along one circle. This can be done by diagonalizing the (skew-Hermitian) matrix with quaternion entries, and take two $2\times 2$ submatrices with one common entry which contributes to the negative inertia index. Then the same proof as above also works in this case.
\end{proof}

\begin{remark}\label{not_geom_finite}
Actually, the nonseparable subgroups constructed in this proof is not geometrically finite, so it is consistent with the result in \cite{BHW} (standard arithmetic hyperbolic manifolds have LERF fundamental groups). For a cocompact lattice $\Lambda<\text{Isom}_+(\mathbb{H}^n)$ and a finitely generated subgroup $H<\Gamma$, $H$ is geometrically finite if and only if $H$ is a quasi-convexity subgroup of $\Lambda$ (\cite{Sw}), which is also equivalent to that the inclusion $H\hookrightarrow \lambda$ is a quasi-isometric embedding (\cite{BGSS}).

In our construction, the nonseparable subgroup $H<\pi_1(M^m)$ is a free product $H\cong H_1*H_2$, where $H_1$ is a fibered surface subgroup of a (immersed) $3$-dimensional totally geodesic submanifold $M_1$ of $M^m$. So if $H\hookrightarrow \pi_1(M^m)$ is a quasi-isometric embedding, since $H_1\hookrightarrow H_1*H_2\cong H$ and $\pi_1(M_1)\hookrightarrow \pi_1(M^m)$ are both quasi-isometric embeddings, $H_1\hookrightarrow \pi_1(M_1)$ must also be a quasi-isometric embedding. However, it is impossible since fibered surface subgroups of hyperbolic $3$-manifold groups have exponential distortion.

\begin{diagram}
H_1 &\rTo & H\cong H_1*H_2\\
\dTo & &\dTo\\
\pi_1(M_1)&\rTo &\pi_1(M^m)
\end{diagram}
\end{remark}

In \cite{GPS}, \cite{Ag2} and \cite{BT}, the authors did cut-and-past surgery on standard arithmetic hyperbolic manifolds along codimension-$1$ totally geodesic arithmetic submanifolds, and constructed many nonarithmetic hyperbolic manifolds.

In \cite{GPS}, the authors took two non-commensurable standard arithmetic hyperbolic $m$-manifolds, cut them along isometric codimension-$1$ totally geodesic submanifolds, then glue them together by another way. This process is called "interbreeding". In \cite{Ag2} and \cite{BT}, the authors cut one standard arithmetic hyperbolic $m$-manifold along two isometric codimension-$1$ totally geodesic submanifolds, then glue it back in a different way. This process is called "inbreeding", which is first carried out in \cite{Ag2} for $4$-dimensional case, and then generalized to higher dimensions in \cite{BT}.

Since all manifolds constructed in \cite{GPS}, \cite{Ag2} and \cite{BT} contain codimension-$1$ totally geodesic arithmetic submanifolds, we have the following direct corollary of Theorem \ref{arithmetic}.

\begin{theorem}\label{cutpaste}
If $M^m$ is a nonarithmetic hyperbolic $m$-manifold constructed in \cite{GPS} or \cite{BT}, with $m\geq 6$, then $\pi_1(M)$ is not LERF.

Moreover, if $M$ is closed, there exists a nonseparable subgroup isomorphic to a free product of surface groups and free groups. If $M$ is not closed, the nonseparable subgroup is isomorphic to either a free subgroup or a free product of surface groups and free groups.
\end{theorem}

Another geometric way to construct hyperbolic $m$-manifolds is the reflection group method. Suppose $P$ is a finite volume polyhedron in $\mathbb{H}^m$ such that any two codimension-$1$ faces that intersect with each other have dihedral angle $\frac{\pi}{n}$ with integer $n\geq 2$. Then the group generated by reflections along codimension-$1$ faces of $P$ is a discrete subgroup of $\text{Isom}(\mathbb{H}^m)$ with finite covolume.

For any torsion-free finite index subgroup of a reflection group consisting of orientation preserving isometries, the quotient of $\mathbb{H}^m$ is a finite volume hyperbolic $m$-manifold $M$. $M$ is a closed manifold if and only if $P$ is compact. The hyperbolic manifolds constructed by this method are not necessarily arithmetic, and it is known that there exist closed nonarithmetic reflection hyperbolic manifolds with dimension $\leq 5$, and noncompact nonarithmetic reflection hyperbolic manifolds with dimension $\leq 10$ (\cite{VS} Chapter 6.3.2).

When $m\geq 5$, it is easy to see that $M$ still contains two totally geodesic $3$-dimensional submanifolds intersecting along a closed geodesic. To get such a picture, we take two totally geodesic $3$-dimensional planes in $\mathbb{H}^m$ that contain two $3$-dimensional faces of $P$ and intersect with each other along one edge of $P$. Then their images in $M$ are two immersed totally geodesic $3$-dimensional submanifolds. Similarly, for any $m\geq 4$, noncompact reflection hyperbolic $m$-manifolds also have noncompact totally geodesic $3$-dimensional submanifolds.

So we get the following theorem for finite volume hyperbolic manifolds arised from reflection groups. The proof is exactly same as the proof of Theorem \ref{arithmetic} and Theorem \ref{noncompact}.

\begin{theorem}\label{reflection}
Let $M$ be a closed hyperbolic $m$-manifold such that $m\geq 5$, or a noncompact finite volume hyperbolic $m$-manifold with $m\geq 4$. If $\pi_1(M)$ is commensurable with the reflection group of some finite volume polyhedron in $\mathbb{H}^m$, then $\pi_1(M)$ is not LERF.

Moreover, If $M$ is closed, the nonseparable subgroup is isomorphic to a free product of surface groups and free groups. If $M$ is noncompact, there exists a nonseparable subgroup isomorphic to a free group, and another nonseparable subgroup isomorphic to a surface subgroup.
\end{theorem}

By the dimension reason, there are no $\pi_1$-injective $M_1\cup_{\gamma}M_2$ submanifold in a $4$-dimensional (arithmetic) hyperbolic manifold, so Theorem \ref{singular} does not give us any nonLERF fundamental group in dimension $4$. Actually, the nonLERFness of $4$-dimensional closed arithmetic hyperbolic manifold groups is proved in the author's more recent work \cite{Sun}.

\section{Further questions}\label{further}

In this section, we raise a few questions related to the results in this paper.

1. In Remark \ref{engulf}, we get that, for any mixed $3$-manifold $M$, there exist a finite cover $M'$ of $M$ and a $\pi_1$-injective properly immersed subsurface $\Sigma\looparrowright M'$, such that $\pi_1(\Sigma)$ is not contained in any finite index subgroup of $\pi_1(M')$. We may ask whether taking this finite cover is necessary.
\begin{question}
For any mixed $3$-manifold $M$, whether there exists a $\pi_1$-injective (properly) immersed subsurface $\Sigma\looparrowright M$, such that $\pi_1(\Sigma)$ is not contained in any finite index subgroup of $\pi_1(M)$?
\end{question}

\vspace{3mm}

2. None of the results in this paper cover the case of compact (arithmetic) hyperbolic $4$-manifolds, since they neither contain $M_1\cup_{\gamma} M_2$ as a singular submanifold, nor contain a $\mathbb{Z}^2$ subgroup (or a mixed $3$-manifold group as its subgroup).

One possible approach for compact (arithmetic) hyperbolic $4$-manifolds is to study the group of $M_1\cup_{S}M_2$ with $M_1$ and $M_2$ being compact arithmetic hyperbolic $3$-manifolds, with $S$ being a hyperbolic surface embedded in both $M_1$ and $M_2$. In this case, the edge group is a closed surface group, which is much more complicated than $\mathbb{Z}$ or $\mathbb{Z}^2$. The method in this paper seems do not work directly in this case. Even if it works (under some clever modification), the nonseparable (finitely generated) subgroup constructed by this method would be infinitely presented.

This question is actually solved in the author's more recent work \cite{Sun}, which is built on the constructions in this paper.

\vspace{3mm}

3. Given the nonLERFness results of high dimensional (arithmetic) hyperbolic manifolds in this paper, maybe it is not too ambitious to ask the following question about general high dimensional hyperbolic manifolds.

\begin{question}
Whether all finite volume hyperbolic manifolds with dimension at least $4$ have nonLERF fundamental groups?
\end{question}

The main difficulty is that we do not have many examples of finite volume high dimensional hyperbolic manifolds. To the best of the authors knowledge, the main methods for constructing high dimensional hyperbolic manifolds (with dimension $\geq 4$) are: the arithmetic method, the interbreeding and inbreeding method and the reflection group method. In this paper, it is shown in Theorem \ref{arithmetic} and Theorem \ref{noncompact}, Theorem \ref{cutpaste}, and Theorem \ref{reflection} that these three constructions give nonLERF fundamental groups in dimension $\geq 5$ (not $7$-dimensional sporadic examples), $\geq 6$ and $\geq 5$ respectively. Besides these methods, there are other constructions of high dimensional hyperbolic manifolds that invoke some specific right-angled hyperbolic polytopes, e.g. \cite{Da}, \cite{RT}, \cite{ERT1}, \cite{ERT2}, \cite{KoM}. The hyperbolic manifolds obtained by these constructions also contain many totally geodesic $3$-dimensional submanifolds. If the dimension is at least $5$, Theorem \ref{arithmetic} implies that these manifolds have nonLERF fundamental groups, and \cite{Sun} confirms the nonLERFness when the dimension equals $4$.

However, it is difficult to understand a general high dimensional hyperbolic manifold, if we do not assume it lies in one of the above families. The author does not know whether a general high dimensional hyperbolic manifold group contains $3$-manifold subgroups. Maybe a generalization of \cite{KaM} (which shows that each closed hyperbolic $3$-manifold admits a $\pi_1$-injective immersed almost totally geodesic closed subsurface) can do this job, but it seems to be very difficult.

\vspace{3mm}

4. The author expects the method in this paper can be used to prove more groups are not LERF. However, since the author does not have very broad knowledge in group theory, we only consider groups of finite volume hyperbolic manifolds in this paper, which is one of the author's favorite family of groups.

The author also expects the method in this paper can be translated to a purely algebraic proof, instead of a geometric one. Actually, most part of the proof are essentially algebraic, except for one point. In Proposition \ref{surface} and Proposition \ref{nonseparable} (also Proposition \ref{singularsurface} and Proposition \ref{singularnonseparable}), although the essential part that gives nonseparability is $\Sigma_{1,1}\cup \Sigma_{2,1}$, we still need to take a bigger (singular) surface so that it defines a nontrivial $1$-dimensional cohomology class in some finite cover. Then we take a finite cyclic cover dual to this cohomology class and get a contradiction. Although this process seems can be done algebraically, the author does not know how to work it out.


\begin{thebibliography}{ZZZZ}

\bibitem[Ag1]{Ag1} {I.~Agol}, {\it Tameness of hyperbolic $3$-manifolds}, preprint 2004, \url{https://arxiv.org/abs/math/0405568}.

\bibitem[Ag2]{Ag2} {I.~Agol}, {\it Systoles of hyperbolic $4$-manifolds}, preprint 2006, \url{http://arxiv.org/abs/math/0612290}.


\bibitem[Ag3]{Ag3} {I.~Agol}, {\it The virtual Haken conjecture}, with an appendix by I.~Agol, D.~Groves, J.~Manning, Documenta Math. 18 (2013), 1045 - 1087.

\bibitem[BGSS]{BGSS} {G.~Baumslag, S.~M.~Gersten, M.~Shapiro, H.~Short}, {\it Automatic groups and amalgams}, J.~Pure Appl.~Algebra 76 (1991), no. 3, 229 - 316.

\bibitem[BT]{BT} {M.~Belolipetsky, S.~Thomson}, {\it Systoles of hyperbolic manifolds}, Algebr.~Geom.~Topol. 11 (2011), no. 3, 1455 - 1469.

\bibitem[BHW]{BHW} {N.~Bergeron, F.~Haglund, D.~Wise}, {\it Hyperplane sections in arithmetic hyperbolic manifolds}, J.~Lond.~Math.~Soc. (2) 83 (2011), no. 2, 431 - 448.

\bibitem[Bo]{Bo} {B.~H.~Bowditch}, {\it Geometrical finiteness for hyperbolic groups}, J.~Funct.~Anal. 113 (1993), no. 2, 245 - 317.

\bibitem[BBS]{BBS} {A.~Brunner, R.~Burns, D.~Solitar}, {\it The subgroup separability of free products of two free groups with cyclic amalgamation}, Contributions to group theory, 90 - 115, Contemp.~Math., 33, Amer.~Math.~Soc., Providence, RI, 1984.

\bibitem[CG]{CG} {D.~Calegari and D.~Gabai}, {\it Shrinkwrapping and the taming of hyperbolic $3$-manifolds}, J. Amer. Math. Soc. 19 (2006), no. 2, 385 - 446.

\bibitem[Ca]{Ca} {R.~Canary}, {\it A covering theorem for hyperbolic $3$-manifolds and its applications}, Topology 35 (1996), no. 3, 751 - 778.


\bibitem[CDW]{CDW} {E.~Chesebro, J.~DeBlois, H.~Wilton}, {\it Some virtually special hyperbolic $3$-manifold groups}, Comment.~Math.~Helv. 87 (2012), no. 3, 727 - 787.

\bibitem[Da]{Da} {M.~Davis}, {\it A hyperbolic $4$-manifold}, Proc.~Amer.~Math.~Soc. 93 (1985), no. 2, 325 - 328.

\bibitem[DLW]{DLW} {P.~Derbez, Y.~Liu, S.~Wang}, {\it Chern--Simons theory, surface separability, and volumes of $3$-manifolds}, J. Topol. 8 (2015), no. 4, 933 - 974.

\bibitem[ERT1]{ERT1} {B.~Everitt, J.~Ratcliffe, S.~Tschantz}, {\it The smallest hyperbolic $6$-manifolds}, Electron.~Res.~Announc.~Amer.~Math.~Soc. 11 (2005), 40 - 46.

\bibitem[ERT2]{ERT2} {B.~Everitt, J.~Ratcliffe, S.~Tschantz}, {\it Right-angled Coxeter polytopes, hyperbolic six-manifolds, and a problem of Siegel},
Math.~Ann. 354 (2012), no. 3, 871 - 905.

\bibitem[Gi]{Gi} {R.~Gitik}, {\it Graphs and separability properties of groups}, J.~Algebra 188 (1997), no. 1, 125 - 143.

\bibitem[GPS]{GPS} {M.~Gromov, I.~Piatetski-Shapiro}, {\it Nonarithmetic groups in Lobachevsky spaces}, Inst.~Hautes \'{E}tudes Sci. Publ. Math. No. 66 (1988), 93 - 103.

\bibitem[Ha]{Ha} {M.~Hall}, {\it Coset representations in free groups}, Trans.~Amer.~Math.~Soc. 67 (1949), 421 - 432.

\bibitem[KaM]{KaM} {J.~Kahn, V.~Markovic}, {\it Immersing almost geodesic surfaces in a closed hyperbolic three manifold}, Ann. of Math. (2) 175 (2012), no. 3, 1127 - 1190.

\bibitem[KoM]{KoM} {A.~Kolpakov, B.~Martelli}, {\it Hyperbolic four-manifolds with one cusp}, Geom. Funct. Anal. 23 (2013), no. 6, 1903 - 1933.

\bibitem[Ka]{Ka} {M.~Kapovich}, {\it Noncoherence of arithmetic hyperbolic lattices}, Geom.~Topol. 17 (2013), no. 1, 39 - 71.

\bibitem[LM]{LM} {J.~Li, J.~Millson}, {\it On the first Betti number of a hyperbolic manifold with an arithmetic fundamental group}, Duke Math.~J. 71 (1993), no. 2, 365 - 401.

\bibitem[Li1]{Li1} {Y.~Liu}, {\it A characterization of virtually embedded subsurfaces in $3$-manifolds}, Trans.~Amer.~Math.~Soc. 369 (2017), no. 2, 1237 - 1264.

\bibitem[Li2]{Li2} {Y.~Liu}, {\it Erratum to "A characterization of virtually embedded subsurfaces in $3$-manifolds"}, Trans.~Amer.~Math.~Soc. 369 (2017), no. 2, 1513 - 1515.

\bibitem[LR]{LR} {D.~Long, A.~Reid}, {\it The fundamental group of the double of the figure-eight knot exterior is GFERF}, Bull.~London Math.~Soc. 33 (2001), no. 4, 391 - 396.

\bibitem[Ma]{Ma} {B.~Maskit}, {\it Kleinian groups}, Grundlehren der Mathematischen Wissenschaftenm, 287, Springer - Verlag, Berlin, 1988. xiv+326 pp.

\bibitem[Ni1]{Ni1} {G.~Niblo}, {\it The subgroup separability of amalgamated free products}, Ph.D. Thesis, University of Liverpool, 1988.

\bibitem[Ni2]{Ni2} {G.~Niblo}, {\it H.N.N. extensions of a free group by $\mathbb{Z}$ which are subgroup separable}, Proc.~London Math.~Soc. (3) 61 (1990), no. 1, 18 - 32.

\bibitem[NW1]{NW1} {G.~Niblo, D.~Wise}, {\it The engulfing property for $3$-manifolds}, The Epstein birthday schrift, 413 - 418 (electronic), Geom.~Topol.~Monogr., 1, Geom.~Topol.~Publ., Coventry, 1998.

\bibitem[NW2]{NW2} {G.~Niblo, D.~Wise}, {\it Subgroup separability, knot groups and graph manifolds}, Proc. Amer. Math. Soc. 129 (2001), no. 3, 685 - 693.

\bibitem[PW1]{PW1} {P.~Przytycki, D.~Wise}, {\it Mixed $3$-manifolds are virtually special}, J.~Amer.~Math.~Soc. 31 (2018), no. 2, 319 - 347.

\bibitem[PW2]{PW2} {P.~Przytycki, D.~Wise}, {\it Separability of embedded surfaces in $3$-manifolds}, Compos.~Math. 150 (2014), no. 9, 1623 - 1630.

\bibitem[RT]{RT} {J.~Ratcliffe, S.~Tschantz}, {\it The volume spectrum of hyperbolic 4-manifolds}, Experiment.~Math. 9 (2000), no. 1, 101 - 125.

\bibitem[Ri]{Ri} {E.~Rips}, {\it An example of a non-LERF group which is a free product of LERF groups with an amalgamated cyclic subgroup}, Israel J.~Math. 70 (1990), no. 1, 104 - 110.

\bibitem[RW]{RW} {H.~Rubinstein, S.~Wang}, {\it $\pi_1$-injective surfaces in graph manifolds}, Comment. Math. Helv. 73 (1998), no. 4, 499 - 515.

\bibitem[Sc]{Sc} {P.~Scott}, {\it Subgroups of surface groups are almost geometric}, J.~London Math.~Soc. (2) 17 (1978), no. 3, 555 - 565.

\bibitem[Sw]{Sw} {G.~A.~Swarup}, {\it Geometric finiteness and rationality}, J.~Pure Appl.~Algebra 86 (1993), no. 3, 327 - 333.

\bibitem[Sun]{Sun} {H.~Sun}, {\it Geometrically finite amalgamations of hyperbolic 3-manifold groups are not LERF}, preprient 2017, \url{https://arxiv.org/abs/1705.03498}.

\bibitem[Th1]{Th1} {W.~Thurston}, {\it The geometry and topology of three-manifolds}, Princeton lecture notes, 1979,
available at \url{http://www.msri.org/publications/books/gt3m/}.


\bibitem[Th2]{Th2} {W.~Thurston}, {\it A norm for the homology of $3$-manifolds}, Mem.~Amer.~Math.~Soc. 59 (1986), no. 339, i - vi and 99 - 130.

\bibitem[Th3]{Th3} {W.~Thurston}, {\it On the geometry and dynamics of diffeomorphisms of surfaces}, Bull.~Amer.~Math.~Soc. 19 (1988), no. 2, 417 - 431.

\bibitem[VS]{VS} {E.~Vinberg, O.~Shvartsman}, {\it Discrete groups of motions of spaces of constant curvature}, Geometry, II, 139 - 248,
Encyclopaedia Math. Sci., 29, Springer, Berlin, 1993.

\bibitem[Wi]{Wi} {D.~Wise}, {\it From riches to raags: $3$-manifolds, right-angled Artin groups, and cubical geometry}, CBMS Regional Conference Series in Mathematics, 117.


\end{thebibliography}
\end{document}